\documentclass[11pt]{article}
\usepackage{geometry}
\usepackage{sty}
\usepackage{am}

\date{}

\title{Equivalence of Approximate Message Passing and\\ Low-Degree Polynomials in Rank-One Matrix
Estimation}

\author{
Andrea Montanari\thanks{Department of Electrical Engineering
    and Department of Statistics, Stanford University; School of Mathematics,
     Institute for Advanced Studies, Princeton} \and 
   Alexander S.\ Wein\thanks{Department of Mathematics, University of California, Davis}}
    
\begin{document}

\maketitle

\begin{abstract}
We consider the problem of estimating an unknown parameter vector 
$\btheta\in\reals^n$, given noisy observations $\bY = \btheta\btheta^{\sT}/\sqrt{n}+\bZ$
of the rank-one matrix $\btheta\btheta^{\sT}$, where $\bZ$ has independent Gaussian entries.
When information is available about the distribution of the entries of $\btheta$, 
spectral methods are known to be strictly sub-optimal. Past work characterized the asymptotics of
the accuracy achieved by the optimal estimator. However, no
polynomial-time estimator is known that achieves this accuracy.

It has been conjectured that this statistical-computation gap is fundamental, and
moreover that the optimal accuracy achievable by polynomial-time estimators coincides with 
the accuracy achieved by certain approximate message passing (AMP) algorithms.
We provide evidence towards this conjecture by proving that no estimator in
the (broader) class of constant-degree polynomials can surpass AMP.
\end{abstract}

\section{Introduction}
\label{sec:Introduction}

Statistical-computational gaps are a ubiquitous phenomenon in high-dimensional statistics.
Consider estimation of a high-dimensional parameter vector $\btheta= (\theta_1,\theta_2,\dots,\theta_n)$
from noisy observations $\bY\sim {\rm P}_{\btheta}$. In many models, we can characterize
the optimal accuracy achieved by arbitrary estimators. However, when we analyze classes
of estimators that can be implemented via polynomial-time algorithms, a significantly
smaller accuracy is obtained. A short list of examples
includes sparse regression~\cite{celentano2022fundamental,gamarnik2022sparse}, 
sparse principal component analysis~\cite{johnstone2009consistency,amini2009high,krauthgamer2015sdp,BR-reduction},
graph clustering and community detection~\cite{decelle2011asymptotic,barak2019nearly}, tensor principal component analysis~\cite{richard2014statistical,hopkins2015tensor}, and tensor decomposition~\cite{MSS-tensor,W-tensor}.

This state of affairs has led to the conjecture that these gaps are fundamental in
nature: there simply exists no polynomial-time algorithm achieving the statistically
optimal accuracy. Proving such a conjecture is extremely challenging since 
standard complexity theoretic assumptions (e.g.\ P$\neq$NP) are ill-suited to
establish complexity lower bounds in \emph{average-case} settings where the input to the algorithm is random. A possible approach to overcome this challenge is
to establish `average-case' reductions between statistical estimation problems.
We refer to \cite{brennan2018reducibility,brennan2020reducibility} for pointers
to this rich line of work.

A second approach is to prove reductions \emph{between classes of polynomial-time algorithms},
thus trimming the space of possible algorithmic choices. This paper
contributes to this line of work by establishing---in a specific context---that
approximate message passing (AMP) algorithms and low-degree (Low-Deg) polynomial
estimators achieve asymptotically the same accuracy. 

Examples of reductions between algorithm classes in statistical estimation include 
\cite{hopkins2017power,brennan2020statistical,celentano2020estimation,banks2021local,celentano2022fundamental}. 
The motivation for studying the relation between AMP
and Low-Deg comes from the distinctive position occupied by these two classes in the theoretical 
landscape.
AMP are iterative algorithms motivated by ideas in information theory 
(iterative decoding~\cite{gallager1962low,RiU08}) 
and statistical physics (cavity method and TAP equations~\cite{thouless1977solution,SpinGlass}). 
Their structure is
relatively constrained: they operate by alternating a matrix-vector multiplication
(with a matrix constructed from the data $\bY$), and a non-linear operation on the same vectors
(typically independent of the data matrix).
This structure has opened the way to sharp characterization
of their behavior in the high-dimensional limit, known as `state evolution'~\cite{DMM09,BM-MPCS-2011,bolthausen2014iterative}.

 Low-Deg was originally motivated by a connection to Sum-of-Squares (SoS)
semidefinite programming relaxations and captures a different (algebraic) notion
of complexity \cite{HS-bayesian,hopkins2017power,schramm2022computational}. In Low-Deg procedures, each unknown parameter
$\theta_i$ is estimated via a fixed polynomial in the entries of the data matrix $\bY$, of constant or 
moderately growing degree. As such, these estimators are relatively unconstrained, and indeed capture 
(via polynomial approximation) a broad variety of natural approaches. 
Developing a sharp analysis of such a broad class of estimators can be very challenging.

In the symmetric rank-one estimation problem, we observe a noisy version $\bY$ of the
rank-one matrix $\btheta\btheta^{\sT}$, with $\btheta\in\reals^n$ an unknown vector:
\begin{align}
\bY = \frac{1}{\sqrt{n}}\btheta\btheta^{\sT}+\bZ\, .\label{eq:ModelDef}
\end{align}
Here $\bZ$ is a random matrix independent of $\btheta$, drawn from the Gaussian Orthogonal 
Ensemble $\bZ\sim \GOE(n)$ which we define here by $\bZ=\bZ^{\sT}$ and $(Z_{ij})_{i\le j}$ independent with $Z_{ii}\sim \normal(0,2)$ and $Z_{ij}\sim \normal(0,1)$ for $i<j$.
Given a single observation of the matrix $\bY$, we would like to estimate 
$\btheta$.

Because of its simplicity, the rank-one estimation problem has been a useful
sandbox to develop new mathematical ideas in high-dimensional statistics. A significant
amount of work has been devoted to the analysis of spectral estimators, which typically
take the form $\hbtheta(\bY) =c_n\, \bv_1(\bY)$ where $\bv_1(\bY)$ is the top eigenvector of $\bY$
and $c_n$ is a scaling factor~\cite{hoyle2004principal,baik2005phase,baik2006eigenvalues,benaych2011eigenvalues}. 
However, spectral methods are known to 
be suboptimal if
 additional information is available about the entries $\theta_i$ of $\theta$. 
 In this paper, we model this
 information by assuming $(\theta_i)_{1\le i\le n}\sim_{iid} \pi_{\Theta}$, for $\pi_{\Theta}$
 a probability distribution on $\reals$ (which does not depend on $n$). The resulting Bayes error coincides (up to terms negligible as $n\to\infty$)
 with the minimax optimal error in the class of vectors with empirical distribution 
 converging to $\pi_{\Theta}$. Hence this model captures the minimax behavior in a well-defined class of signal parameters.

In the high-dimensional limit $n\to\infty$ (with $\pi_\Theta$ fixed), the Bayes optimal accuracy for estimating $\btheta$ under
model \eqref{eq:ModelDef}  (and the above assumptions on $\btheta$) is known to converge to a
 well-defined limit that was characterized in a sequence of 
 papers, see e.g.~\cite{lelarge2019fundamental,barbier2019adaptive,chen2022hamilton}.
The asymptotic accuracy of the optimal AMP algorithm (called Bayes AMP) is also known~\cite{montanari2021estimation}. 

It is useful to pause and recall these results in some detail.
Define the  function $\Psi: \reals_{\ge 0}\times \reals \times \cuP(\reals)\to \reals$  (here
and below $\cuP(\reals)$ denotes the set of probability distributions over $\reals$) 
by letting
\begin{align}
	 \Psi(q;b,\pi_{\Theta}) &:= \frac{1}{4} q^2-
	 \frac{1}{2}\big(\E[\Theta^2]+b\big) q+\Info(q;\pi_{\Theta}) \, ,\label{eq:PsiDef}\\
	 \Info(q;\pi_{\Theta}) &:= \E\log \frac{\de p_{\Yeff|\Theta}}{\de p_{\Yeff}}\,,
	  \;\;\;\;\;
	  \Yeff = \sqrt{q}\Theta +G \, ,\;\;\; (\Theta,G)\sim\pi_{\Theta}\otimes \cN(0,1)\, .	 
\end{align}
 Note that  $\Info(q;\pi_{\Theta})$ can be interpreted as the mutual information between a scalar
 random variable $\Theta\sim \pi_{\Theta}$ and the scalar noisy observation $\Yeff$.
 It can be expressed as a one-dimensional integral with respect to $\pi_{\Theta}$:
 \begin{align}
 \Info(q;\pi_{\Theta}) &:=  -\E\log \Big\{\int e^{-(\Yeff-\sqrt{q}\theta)^2/2}
 \,\pi_{\Theta}(\de\theta)\Big\}-\frac{1}{2} \, .
 \end{align}
The next statement adapts results from~\cite{lelarge2019fundamental} 
 concerning the behavior of the Bayes optimal error (see Appendix \ref{app:IT-result}, which details the derivation from 
 \cite{lelarge2019fundamental}).
 A formal definition of Bayes AMP will be given in the next section,
 alongside a formal definition of Low-Deg algorithms. 
 \begin{theorem}\label{thm:Earlier}
Assume $\pi_{\Theta}$ is independent of $n$, has non-vanishing first moment $\E[\Theta]\neq 0$, and has finite moments of all orders. Define
 \begin{align}
 q_{\sBayes}(\pi_{\Theta}) &:= \argmin_{q\ge 0}\, \Psi(q;0,\pi_{\Theta})\, ,\\
 q_{\sAMP}(\pi_{\Theta}) & :=  \inf\Big\{q\ge 0 \;:\; \Psi'(q;0,\pi_{\Theta}) =0 ,\; 
 \Psi''(q;0,\pi_{\Theta}) \ge 0\Big\}\, .
 \end{align}
Bayes AMP has time complexity $O(c(n)\,n^2)$, for any diverging sequence $c(n)$,
and achieves mean squared error (MSE)
 \begin{align}
 \lim_{n\to\infty}\frac{1}{n}\E\big\{\|\hbtheta^{\sAMP}(\bY)-\btheta\|^2_2\big\}
  = \E[\Theta^2]- q_{\sAMP}(\pi_{\Theta})\, .
\end{align}
 Further, assume that $b\mapsto \Psi_*(b;\pi_{\Theta}):=\min_{q\ge 0} \Psi(q;b,\pi_{\Theta})$ is
 differentiable at $b=0$.
 Then the minimum MSE of any estimator is
 \begin{align}
 \lim_{n\to\infty}\inf_{\hbtheta(\,\cdot\,)}\frac{1}{n}\E\big\{\|\hbtheta(\bY)-\btheta\|^2_2\big\}
  = \E[\Theta^2]- q_{\sBayes}(\pi_{\Theta})\, .\label{eq:AsympBayes}
\end{align}
 \end{theorem}
 \begin{remark}[Differentiability at $b=0$]
 The function $b\mapsto \Psi_*(b;\pi_{\Theta})$
 is concave because it is a minimum of linear functions, and therefore differentiable at all except for
 countably many values of $b$. Hence the differentiability assumption amounts to requiring that $b=0$ is non-exceptional.
 
 Also note that replacing the prior $\pi_{\Theta}$ by its shift $\pi_{\Theta'}$
 where $\Theta'=\Theta+a$ amounts to changing $b$ to $b'=b+2\E[\Theta]a+a^2$.
 Hence the assumption that $b=0$ is non-exceptional is equivalent to the assumption that the
  mean of $\pi_{\Theta}$ is non-exceptional.
 
 We also note that our results will not require such genericity assumptions, 
 which we only introduce to offer a simple comparison to the Bayes optimal estimator.
\end{remark}

\begin{remark}[Nonzero mean assumption]
The assumption $\E[\Theta]\neq 0$ can be removed from Theorem \ref{thm:Earlier}
provided the mean squared error metric is replaced by a different metric that
is invariant under change of relative sign in $\btheta$ and $\hbtheta$.
In that case,
Bayes AMP needs to be modified, e.g.\ via a spectral initialization \cite{montanari2021estimation}.
Here we focus on the case $\E[\Theta]\neq 0$ because this is the most relevant for
the results that follow.
\end{remark}

In this paper we compare Bayes AMP run for a constant number $t=O(1)$ 
of iterations and Low-Deg estimators with degree $D=O(1)$. Assuming $\E[\Theta]\neq 0$ and $\pi_{\Theta}$ is sub-Gaussian, we establish the following (see Theorem~\ref{thm:Main}):
\begin{itemize}
\item For any fixed $t$ and $\eps>0$, there exists a constant $D=D(t,\eps)$ and 
a degree-$D$ estimator that approximates the MSE achieved by Bayes AMP after $t$ iterations within an additive 
error of at most $\eps$.
\item For any constant $D=O(1)$, no degree-$D$ estimator can surpass the asymptotic MSE of Bayes AMP.
Namely, for any fixed $D$ and any degree-$D$  estimator $\hbtheta$, we have
$\E\big\{\|\hbtheta(\bY)-\btheta\|^2_2\big\}/n \ge \E[\Theta^2]- q_{\sAMP}(\pi_{\Theta})-o_n(1)$.
\end{itemize}
Here and throughout, the notation $o_n(1)$ signifies a quantity that vanishes in the limit where $n \to \infty$ with all other parameters (such as $D$) held fixed. The first claim above (`upper bound') is proved by a straightforward polynomial approximation of Bayes AMP. To obtain the second claim (`lower bound'), we develop a new proof technique. Given a Low-Deg estimator for coordinate $i$, we express $\htheta_i(\bY)$ as a sum of terms that
are associated to rooted multi-graphs on vertex set $[n] := \{1,2,\ldots,n\}$ with at most $D$ edges. We group these terms into a constant number of homomorphism classes. 
We next prove that among these classes, only those corresponding to 
 trees yield a non-negligible contribution 
as $n\to\infty$. Finally we show that the latter contribution can be encoded into an 
AMP algorithm and use existing optimality theory for AMP algorithms~\cite{celentano2020estimation,montanari2022statistically} to deduce the result.

This new proof strategy elucidates the relation between AMP 
and Low-Deg algorithms: roughly speaking, AMP algorithms correspond to `tree-structured' low-degree
polynomials, a subspace of all Low-Deg estimators. AMP can effectively evaluate tree-structured polynomials via a dynamic-programming
style recursion with complexity $O(n^2\cdot \mathsf{depth})$ (with $\mathsf{depth}$ the tree depth)
instead of the naive $O(n^{D+1})$. 
 
The rest of the paper is organized as follows. Section \ref{sec:main-res} provides the necessary 
background, formally states our results, and discusses limitations, implications,
and future directions.
Sections \ref{sec:proofUB} and \ref{sec:proofLB} prove the main theorem (Theorem~\ref{thm:Main}),
respectively establishing the upper and lower bounds
on the optimal estimation error achieved by Low-Deg. 
The proofs of several technical lemmas are deferred to the appendices.

\section{Main results}
\label{sec:main-res}

\subsection{Background: AMP}

The  class of AMP algorithms\footnote{More general settings have been studied and analyzed 
in the literature~\cite{javanmard2013state,berthier2020state,gerbelot2021graph,fan2022approximate}.} 
that we will consider in this paper proceed
iteratively by updating a state $\bx^t\in\reals^{n\times \dim}$ according to
the iteration:
\begin{align}
\bx^{t+1} &= \frac{1}{\sqrt{n}}\bY\, F_t(\bx^t) -   F_{t-1}(\bx^{t-1})\sB^{\sT}_t\, ,\;\;\;\; t\ge 0\, .
\label{eq:FirstAMP}
\end{align}
Throughout, we will assume the uninformative initialization $\bx^0=\bzero$, $\sB_0=0$.
In the above $F_t:\reals^{\dim}\to\reals^{\dim}$ is a function which operates on the matrix
$\bx^t\in\reals^{n\times \dim}$ row-wise. Namely for $\bx\in\reals^{n\times \dim}$
with rows $\bx_1^\sT,\dots, \bx_n^{\sT}$, we have
 $F_t(\bx) :=(F_t(\bx_1),\dots,F_t(\bx_n))^{\sT}$.
After $t$ iterations, we estimate the signal $\btheta$ via
$\hbtheta^t :=  g_t(\bx^t)$, for some function $g_t:\reals^{\dim}\to\reals$ (again, applied row-wise).
We will consider $\dim$ and the sequence of functions $F_t$ fixed, while $n\to\infty$.

The sequence of matrices $\sB_t\in\reals^{\dim\times \dim}$ can be taken to be non-random 
(independent of $\bY$)
and will be specified shortly. The high-dimensional asymptotics of the above iteration
is characterized by the following finite-dimensional recursion over variables
$\bmu_t\in\reals^{\dim}$, $\bSigma_t\in\reals^{\dim\times\dim}$, known as `state evolution,'
which is initialized at $(\bmu_0,\bSigma_0)=(\bzero,\bzero)$:
\begin{align}
\bmu_{t+1} &= \E\big\{\Theta F_t(\bmu_t\Theta+\bG_t)\big\}\, ,
\;\;\;\;\; (\Theta,\bG_t)\sim \pi_{\Theta}\otimes \cN(0,\bSigma_t)\label{eq:GeneralSE1}\\
\bSigma_{t+1}& = \E\big\{F_t(\bmu_t\Theta+\bG_t) F_t(\bmu_t\Theta+\bG_t)^{\sT}\big\}\, .
\label{eq:GeneralSE2}
\end{align}
In terms of this sequence, we define $\sB_t$ via
\begin{align}
\sB_t = \E\{\De F_t(\mu_t\Theta+\bG_t)\}\, ,
\end{align}
where $\De F_t = (\partial_iF_{t,j})_{i,j \in [\dim]}$ denotes the weak differential of $F_t$.

The next theorem characterizes the high-dimensional asymptotics of 
the above AMP algorithm. It summarizes results from \cite{BM-MPCS-2011,bayati2015universality}
(see e.g.~\cite{montanari2022statistically} for the application to rank-one estimation).
\begin{theorem}\label{thm:SE}
Assume $\pi_{\Theta}$ is independent of $n$ and has finite moments of all orders.
Further assume that the functions $\{F_t, g_t\}_{t\ge 0}$ are independent of $n$ and either:
$(i)$~for each $t$, $F_t$ and $g_t$ are $L_t$-Lipschitz, for some $L_t<\infty$; or
$(ii)$~for each $t$, $F_t$ and $g_t$ are degree-$B_t$ polynomials, for some $B_t<\infty$.
Then
\begin{align}
\lim_{n\to\infty} \frac{1}{n}\E\big\{\|\hbtheta^t(\bY)-\btheta\|_2^2\big\}
=\E\big\{\big[\Theta-g_t(\bmu_t\Theta+\bG_t)\big]^2\big\}\, ,\label{eq:FirstThmSE}
\end{align}
where expectation is with respect to $(\Theta,\bG_t)\sim \pi_{\Theta}\otimes \cN(0,\bSigma_t)$.
\end{theorem}

Using the above result, it follows that the optimal function $F_t$
is given by the posterior expectation denoiser. Namely, define the one-dimensional recursion
(with $G\sim\normal(0,1)$):
\begin{align}
q_{t+1} = \E\big\{\E[\Theta \,|\, q_t\Theta+ \sqrt{q_t}G]^2\big\} =:\SE(q_t;\pi_{\Theta}) \, ,\;\;\; q_0=0\, .
\label{eq:BayesSE}
\end{align}
Consider $\dim=1$ in the recursion of Eqs.~\eqref{eq:GeneralSE1}, \eqref{eq:GeneralSE2}, 
so $\mu_t\in\reals$ and $\Sigma_t =: \sigma^2_t\in\reals$, and take
$g_t(x) = F_t(x):= \E[\Theta \,|\, q_t\Theta+ \sqrt{q_t}G=x]$ with $q_t$ defined in Eq.~\eqref{eq:BayesSE}. We have 
\begin{align}
&\mu_t=\sigma^2_t = q_t\, ,\\
& \E\big\{\big[\Theta-g_t(\mu_t\Theta+G)\big]^2\big\} = \E[\Theta^2]-q_{t+1}\, .
\end{align}
The next result establishes that this is indeed the optimal MSE achieved by AMP algorithms.
\begin{theorem}[\cite{montanari2022statistically}]\label{thm:BayesSE}
Assume $\pi_{\Theta}$ is independent of $n$ and has finite second moment.
Then, for any AMP algorithm satisfying the assumptions of Theorem \ref{thm:SE}, we have for any fixed
 $t \ge 0$, 
\begin{align}
\lim_{n\to\infty} \frac{1}{n}\E\big\{\|\hbtheta^t(\bY)-\btheta\|_2^2\big\}
 \ge \E[\Theta^2]-q_{t+1}\, .
\end{align}
Further there exists a sequence of AMP algorithms approaching the lower bound arbitrarily
closely.

Finally, the fixed points of iteration \eqref{eq:BayesSE} (i.e., the points $q$ such that
$q=\SE(q;\pi_{\Theta})$) coincide with the stationary points of $\Psi(q;b=0,\pi_{\Theta})$ 
defined in Eq.~\eqref{eq:PsiDef} (i.e., the points $q$ such that
$\partial_q\Psi(q;0,\pi_{\Theta})=0$).
\end{theorem}

\begin{figure}[t]
\centering
\includegraphics[width=0.7\textwidth]{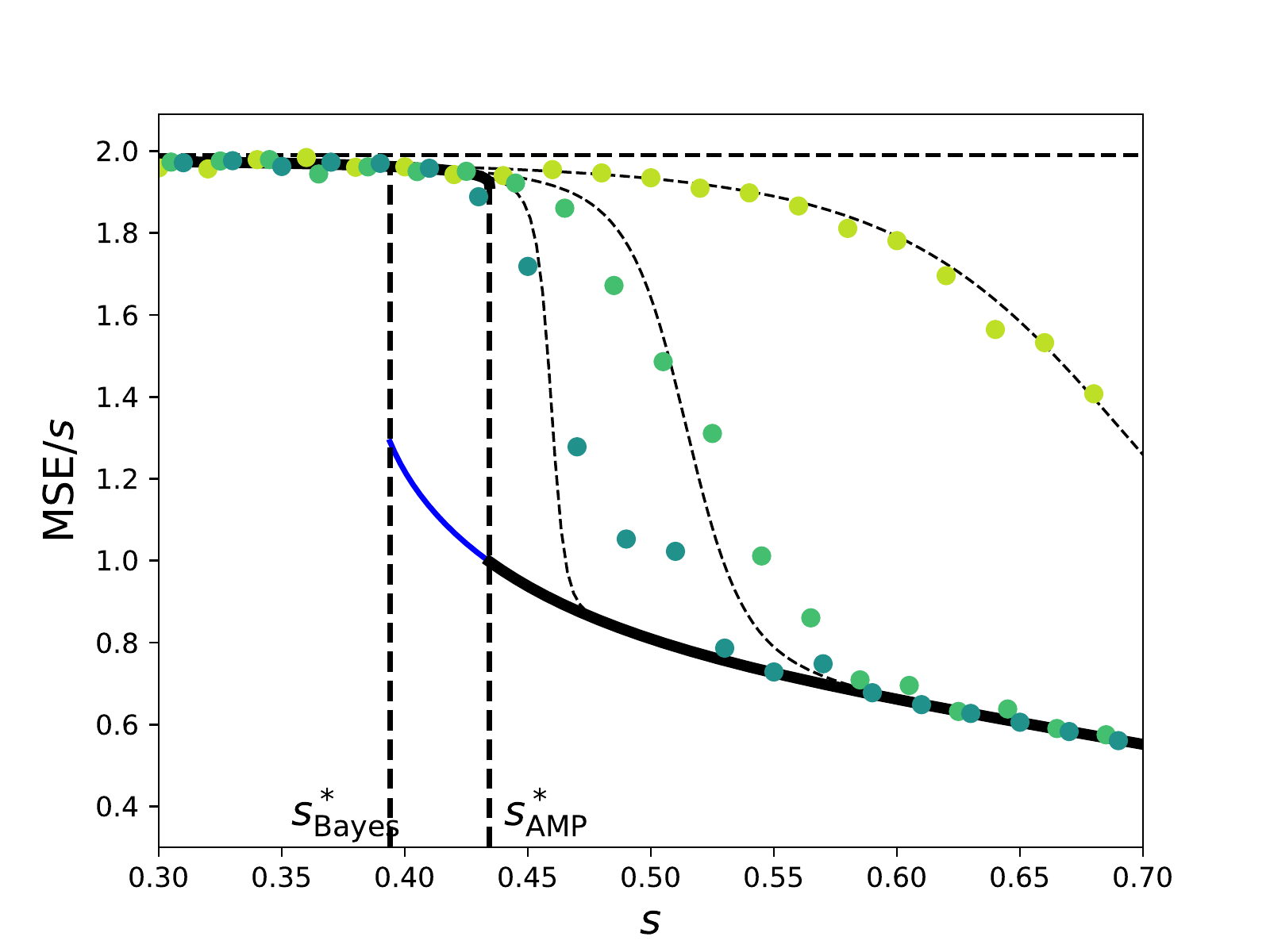} 
\caption{Estimation accuracy in the symmetric rank-one estimation problem~\eqref{eq:ModelDef}, for a one-parameter family of prior
distributions $\pi_{\Theta}^{s}$ defined in Eq.~\eqref{eq:fig-prior}. Solid black curve: asymptotic MSE achieved by Bayes AMP. Blue curve:
Bayes optimal MSE.
Dashed curves: asymptotic MSE for Bayes AMP with
(from bottom to top) $t\in\{5,10,20\}$ iterations.
Circles: average MSE achieved by AMP in simulation for $t\in\{5,10,20\}$.
The black and blue curves coincide outside the interval $[s^*_{\sBayes}, s^*_{\sAMP}]$. Here MSE is normalized by $1/s$ so that the horizontal black dashed line represents the trivial MSE achieved by the constant estimator $\hat\theta_i(\bY) = \EE[\Theta]$.
}\label{fig:amp_vs_bayes}
\end{figure}
As an illustration, Figure \ref{fig:amp_vs_bayes} presents the asymptotic accuracy achieved by Bayes AMP and Bayes optimal estimation, 
as characterized by Theorem \ref{thm:Earlier}.
In this figure we consider the following parametrized family of three-point priors:
\begin{align}\label{eq:fig-prior}
\pi_{\Theta}^s = \frac{1-\eps}{2}\, (\delta_{-\sqrt{s}}+\delta_{+\sqrt{s}}) +\eps\,\delta_{\sqrt{s/\eps}}\, .
\end{align}
We set $\eps=0.01$ and sweep $s$ (which measures the signal-to-noise ratio).
We compare the predicted accuracy for Bayes AMP with numerical simulations for $n=2\cdot 10^4$
and $t\in\{5,10,20\}$ iterations, averaged over $50$ realizations. 

\subsection{Background: Low-Deg}

We say that $\hbtheta:\reals^{n\times n}\to\reals^n$,
$\bY\mapsto \hbtheta(\bY)$ is a Low-Deg estimator of degree $D$ if its coordinates are polynomials of maximum degree $D$ in the matrix entries $(Y_{ij})_{1 \le i\le j\le n}$. The coefficients of these polynomials may depend on $n$ but not on $\bY$.
Let $\RR[\bY]_{\le D}$ denote the space of all polynomials in the variables $(Y_{ij})_{1\le i\le j\le n}$ of degree at most $D$. 
We use $\LD(D;n)$ to denote the set of  estimators whose coordinates
are degree $D$ polynomials:
\begin{align}
\LD(D;n) := \big\{\hbtheta:\reals^{n\times n}\to \reals^n \;:\;\; \forall i, \, \htheta_i \in
\RR[\bY]_{\le D}\big\}.
\end{align}
With an abuse of notation we will often refer to an estimator $\hbtheta$
(e.g.\ a Low-Deg estimator $\hbtheta\in\LD(D;n)$), but really mean a sequence of estimators
$\hbtheta_n$ indexed by the dimension $n$.

\subsection{Statement of main results}

Recall that a random variable $X$ is called \emph{sub-Gaussian} if $\|X\|_{\psi_2} < 0$ where the sub-Gaussian norm is defined as
\begin{equation}\label{eq:subG-norm}
\|X\|_{\psi_2} := \inf\{t > 0 \,:\, \EE \exp(X^2/t^2) \le 2\} \, .
\end{equation}
For example, any distribution with bounded support is sub-Gaussian.

\begin{theorem}\label{thm:Main}
Assume $\pi_{\Theta}$ is independent of $n$ and sub-Gaussian, with $\E[\Theta] \ne 0$. For any $\eps>0$, there exists $D(\eps)<\infty$ and a (family of) estimators
$\hbtheta_{\eps}=\hbtheta_{\eps,n}\in \LD(D(\eps);n)$ such that
\begin{align}
\lim_{n\to\infty} \frac{1}{n}\E\big\{\|\hbtheta_{\eps}(\bY)-\btheta\|_2^2\big\}
\le \E[\Theta^2]-q_{\sAMP}(\pi_{\Theta})+\eps\, .\label{eq:UpperBound}
\end{align}
Further, for any constant $D$, 
\begin{align}
\lim_{n\to\infty} \inf_{\hbtheta\in \LD(D;n)}\frac{1}{n}\E\big\{\|\hbtheta(\bY)-\btheta\|_2^2\big\}
\ge \E[\Theta^2]-q_{\sAMP}(\pi_{\Theta})\,.\label{eq:LowerBound}
\end{align}
\end{theorem} 

The lower bound~\eqref{eq:LowerBound} only requires $\pi_\Theta$ to have all moments finite, 
 rather than sub-Gaussian. The upper bound~\eqref{eq:UpperBound} likely also holds under weaker conditions on $\pi_\Theta$ than sub-Gaussian, but we have not attempted to explore this here. We conclude this summary of our results with a few remarks and directions for future work.

\paragraph{Sharp characterization for Low-Deg.} As a by-product of our results, we obtain a sharp characterization of the optimal accuracy of Low-Deg estimators
with constant degree. To the best of our knowledge, this is the first example 
of such a characterization. Prior work on low-degree estimation~\cite{schramm2022computational} gave bounds on the signal-to-noise ratio (tight up to log factors) under which the low-degree estimation accuracy is either asymptotically perfect or trivial. In our work we are in a setting where the Low-Deg MSE converges to a non-trivial constant, and we pin down the exact constant.

Similarly, our results are sharper than existing computational lower bounds based on the `overlap gap property'~\cite{gamarnik2021overlap}, which rule out a different (incomparable) class of algorithms including certain Markov chains.

\paragraph{Optimality of AMP.} Our results support the conjecture that AMP is asymptotically optimal among computationally-efficient procedures
for certain high-dimensional estimation problems. We are also aware of some problems for which AMP is strictly sub-optimal and has to be modified to capture higher-order correlations
between variables. One such examples is provided by the spiked tensor model 
\cite{richard2014statistical,hopkins2015tensor,wein2019kikuchi}, namely the
generalization of the above model to higher order tensors, where observations take the form
$\bY = n^{-(k-1)/2}\,\btheta^{\otimes k}+\bZ$ for some $k \ge 3$.

We believe that a generalization of the proof techniques developed in this paper can help
distinguish in a principled way which problems can be optimally solved using AMP.
These ideas may also provide guidance to modify AMP for problems 
in which it is sub-optimal. The key properties of the rank-one estimation problem that cause AMP to be optimal among Low-Deg estimators are established in Lemmas~\ref{lem:c-lim} and~\ref{lem:M-lim}; notably these include the block diagonal structure of a certain correlation matrix $\bM_\infty$. As a result of these properties, the best Low-Deg estimator is `tree-structured' and can thus be computed using an AMP algorithm.

\paragraph{Zero-mean prior.} Our result on the equivalence of AMP and Low-Deg actually applies to the case $\E[\Theta]=0$ as well. However, it must be emphasized that
the statement concerns AMP with uninformative initialization and $O(1)$ iterations,
and Low-Deg estimators with $O(1)$ degree. When $\E[\Theta]=0$, the MSE of both these algorithms converges to $\E[\Theta^2]$ as $n\to\infty$: 
it is no better than random guessing.
 
On the other hand, initializing AMP with a spectral initialization $c_n \, \bv_1(\bY)$ (for a suitable scalar $c_n$)
yields the accuracy stated in Theorem \ref{thm:Earlier}, even for $\E[\Theta]=0$. 
Since the leading eigenvector can be approximated by power iteration, we believe it is possible to show that the same accuracy is achievable 
by AMP when run for $O(\log n)$ iterations (and a random initialization),
or by Low-Deg for $O(\log n)$ degree. However generalizing the lower bound of Eq.~\eqref{eq:LowerBound}
to logarithmic degree goes beyond the analysis developed here.
 
\paragraph{Higher degree.} We expect the optimality of Bayes
 AMP to hold within a significantly broader class of low-degree estimators, namely $O(n^c)$-degree estimators for any $c\in (0,1)$. Again, this extension is beyond our current proof
technique. Heuristically, polynomials of degree $O(n^c)$ can be thought of as a proxy for algorithms of runtime $\exp(n^{c + o_n(1)})$ (which is the runtime required to naively evaluate such a polynomial); see e.g.~\cite{ding2019subexponential}.

\section{Proof of Theorem \ref{thm:Main}: Upper bound}
\label{sec:proofUB}

Recall from Theorem~\ref{thm:BayesSE} that the fixed
points of the state evolution recursion \eqref{eq:BayesSE}  coincide with
the stationary points of $\Psi(q,0;\pi_{\Theta})$.
Further, it is easy to see from the definition Eq.~\eqref{eq:BayesSE} that 
$q\mapsto \SE(q;\pi_{\Theta})$ is a non-decreasing function with 
$\SE(0;\pi_{\Theta}) =\E[\Theta]^2>0$.
(This follows from the fact that the minimum mean square error is a non-increasing function
of the signal-to-noise ratio or, equivalently, from Jensen inequality.)
 As a consequence $(q_t)_{t\ge 0}$ is a non-decreasing sequence with  $\lim_{t\to\infty} q_t = q_{\sAMP}$.
Let $t_*$ be such that $q_{t_*+1}\ge q_{\sAMP}-\eps/2$.

Consider the special case of the AMP algorithm of Eq.~\eqref{eq:FirstAMP} with $\dim = 1$, 
with $F_t=f_t:\reals\to\reals$,
and specialize the state evolution recursion  of
Eqs.~\eqref{eq:GeneralSE1}, \eqref{eq:GeneralSE2} to this case. Namely we define
recursively $\mu_s, \sigma_s^2$ for $ s\ge 0$ via
\begin{align}
\mu_{t+1} &= \E\big\{\Theta f_t(\mu_t\Theta+\sigma_t\, G)\big\}\, ,
\;\;\;\;\; (\Theta,G)\sim \pi_{\Theta}\otimes \cN(0,1)\label{eq:Dim1-SE1}\\
\sigma^2_{t+1}& = \E\big\{f_t(\mu_t\Theta+G_t)^2\big\}\, .
\label{eq:Dim1-SE2}
\end{align}
We claim that for $t\ge 0$ the following holds. For any $\eps_0>0$, we 
can construct polynomials $(f_s)_{0\le s\le t}$ of degree $(D_s(\eps_0))_{0 \le s\le t}$ 
(independent of $n$),  we have, 
\begin{align}
\label{eq:ApproxSE}
\big|\mu_t-q_t\big|\le \eps_0\, ,\;\; \big|\sigma^2_t-q_t\big|\le\eps_0\, .
\end{align}
Once this is established, the desired upper bound \eqref{eq:UpperBound} follows by taking $t=t_*+1$ and $\eps_0=\eps/8$.
Consider indeed the AMP estimator $\hbtheta^t(\bY):=f_t(\bx^t)$ where $\bx^t$ is defined by 
the AMP iteration~\eqref{eq:FirstAMP} with $\dim = 1$, $F_t = f_t$
(that is, we choose $g_t=f_t$). 
This estimator is a polynomial of degree $D\le (D_1+1)(D_2+1)\cdots (D_{t_*}+1)$ and,
by Eq.~\eqref{eq:FirstThmSE} (specialized to $\dim=1$), we have
\begin{align}
\lim_{n\to\infty} \frac{1}{n}\E\big\{\|\hbtheta^{t_*}(\bY)-\btheta\|_2^2\big\} &
= \E\big\{\big[\Theta-f_{t_*}(\mu_{t_*}\Theta+\sigma_{t_*} G)\big]^2\big\}\\
&=\E[\Theta^2]-2\mu_{t_*+1}+\sigma_{t_*+1}^2\\
&\le
\E[\Theta^2]-q_{t_*+1}+\frac{\eps}{2}\le \E[\Theta^2]-q_{\sAMP}+\eps\, .
\end{align}

We are left with the task to prove the claim \eqref{eq:ApproxSE}, which we will
do by induction on $t$. The claim holds for $t=0$, and we then assume as an induction hypothesis that it holds for a certain $t$. 
We will denote by $f^{\sBayes}_t(x):= \E[\Theta \,|\, q_t\Theta+ \sqrt{q_t}G=x]$
the ideal nonlinearity.
Now let $\eps_{1,k}\downarrow 0$ be a sequence converging to $0$
and select $f_{s}=f_{k,s}$, $s\le t-1$ according to the induction hypothesis with
$\eps_0=\eps_{1,k}$. We will denote by $\mu_{k,t},\sigma^2_{k,t}$ the corresponding state evolution
quantities. In particular, we can assume without loss of generality
$\mu_{k,t},\sigma^2_{k,t}\le 2q_t$.
 By the state evolution recursion, we have (expectation here is with respect to
 $(\Theta,G)\sim \pi_{\Theta}\otimes\normal(0,1)$)
\begin{align*}
\big|\mu_{k,t+1}-q_{t+1}\big| &=
\Big|\E\Big\{\Theta[f_{k,t}(\mu_{k,t}\Theta+\sigma_{k,t} G) -f^{\sBayes}_t(q_t\Theta+\sqrt{q_t}G)]\Big\}\Big|\\
&\le \Big|\E\Big\{\Theta[f^{\sBayes}_t(\mu_{k,t}\Theta+\sigma_{k,t} G) -
f^{\sBayes}_t(q_t\Theta+\sqrt{q_t}G)]\Big\}\Big|\\
&\phantom{AAAAAAA}+
\Big|\E\Big\{\Theta[f_t(\mu_{k,t}\Theta+\sigma_{k,t} G) -
f^{\sBayes}_t(\mu_{k,t}\Theta+\sigma_{k,t} G)]\Big\}\Big|\\
&=: B_1(k)+B_2(k)\, .
\end{align*}
It is a general fact about Bayes posterior expectation
that $f_t^{\sBayes}$ is continuous with $|f_t^{\sBayes}(x)|\le C_t(1+|x|)$ for some constant $C_t$
(see Lemma \ref{lemma:BayesFact} in Appendix \ref{app:UpperBound}).
Further $\mu_{k,t}\Theta+\sigma_{k,t} G\stackrel{a.s.}{\longrightarrow }q_t\Theta+\sqrt{q_t}G$
as $k\to\infty$. By dominated convergence, we have $\lim_{k\to\infty}B_1(k)=0$ and
therefore we can choose $k_0$ so that for all $k\ge k_0$, $B_1(k)\le \eps_0/2$.

Next consider term $B_2(k)$. Denoting by $\tau := \|\Theta\|_{\psi_2}$ the sub-Gaussian 
norm of $\pi_{\Theta}$ (see Eq.~\eqref{eq:subG-norm}), we have that $Z_{k,t}:=\mu_{k,t}\Theta+\sigma_{k,t} G$
is sub-Gaussian with $\|Z_{k,t}\|_{\psi_2}^2\le \mu_{k,t}^2\tau^2 + \sigma_{k,t}^2\le 
4(q_t^2\tau^2+q_t)$. We let $\tau^2_t:=4(q_t^2\tau^2+q_t)$ denote this upper bound.
Since $f^{\sBayes}_t$ is continuous with $|f^{\sBayes}(x)|\le C_t(1+|x|)$,
by weighted approximation theory \cite{lorentz2005approximation}, we can choose
$f_t$ a polynomial such that
\begin{align}
\sup_{u\in\reals}\big|f_t(u) -
f^{\sBayes}_t(u)\big| \exp\Big(-\frac{u^2}{4\tau^2_t}\Big)\le \frac{\eps_0}{4\E[\Theta^2]^{1/2}}\,.
\end{align}
Using this polynomial approximation, we get
\begin{align}
 B_2(k)&\le \E\big[\Theta^2\big]^{1/2}
\E\big\{[f_t(Z_{k,t}) -f^{\sBayes}_t(Z_{k,t})]^2\big\}^{1/2}\\
& \le \frac{1}{4}\eps_0\E[\exp({Z_{k,t}^2/2\tau_{k,t}^2})]^{1/2}\le \frac{1}{2}\eps_0\, .
 \end{align}
This completes the proof of the induction claim in the first inequality in Eq.~\eqref{eq:ApproxSE}.
The second inequality is treated analogously.

\section{Proof of Theorem \ref{thm:Main}: Lower bound}
\label{sec:proofLB}

Recall that $\RR[\bY]_{\le D}$ denotes the family of polynomials of degree at most
$D$ in the variables $(Y_{ij})_{1\le i\le j\le n}$. The key of our proof is to show
that the asymptotic estimation accuracy of Low-Deg estimators is not reduced
if we replace polynomials by a restricted family of tree-structured polynomials, 
which we will denote by $\RR[\bY]_{\le D}^{\cT}$. 
This reduction is presented in Sections~\ref{sec:Tree1} and~\ref{sec:KeyPropo}. 
We will then establish a connection between
tree-structured polynomials and AMP algorithms, and rely on known optimality 
theory for AMP estimators, cf.\ Section \ref{sec:TreeToAMP}.

\subsection{Reduction to tree-structured polynomials}
\label{sec:Tree1}

Let $\cT_{\le D}$ denote the set of rooted trees up to root-preserving isomorphism,
 with at most $D$ edges. 
  (Trees must be connected, with no self-loops or multi-edges allowed. The tree with one vertex and no edges is included.) We denote the root vertex of  $\Tree \in \cT_{\le D}$ by $\circ$. For $\Tree \in \cT_{\le D}$, define a \emph{labeling rooted at $1$} (or, simply, a labeling) of $\Tree$ to be a function  $\phi: V(\Tree) \to [n]$ such that $\phi(\circ) = 1$. (Vertex 1 is distinguished because we will be considering estimation of $\theta_1$.)
 
Given a graph $G=(V,E)$, we will denote by $\dist_{\Graph}(u,v)$ the graph distance  between two vertices $u,v \in V$.
\begin{definition}\label{def:NonReversing}
A labeling $\phi$ of $\Tree \in \cT_{\le D}$ is said to be \emph{non-reversing} if,
for every pair of distinct vertices $u,v \in V(\Tree)$ with the same label 
(i.e., $\phi(u) = \phi(v)$), one of the following holds:
\begin{itemize}
\item $\dist_{\Tree}(u,v) > 2$, or
\item $\dist_{\Tree}(u,v) = 2$ and $u,v$ have the same distance from the root 
(i.e., $\dist_{\Tree}(u,\circ) = \dist_{\Tree}(v,\circ)$).
\end{itemize}
(The latter holds if and only if $u,v$ are both children of a common vertex $w \in V(\Tree)$.) We denote by  $\nr(\Tree)$  the set of all non-reversing labelings of $\Tree$.
\end{definition}

For each $\Tree \in \cT_{\le D}$, define the associated polynomial
\begin{equation}\label{eq:F_A}
\Poly_{\Tree}(\bY) = \frac{1}{\sqrt{|\nr(\Tree)|}} \sum_{\phi \in \nr(\Tree)} \, 
\prod_{(u,v) \in E(\Tree)} Y_{\phi(u),\phi(v)}.
\end{equation}
\begin{definition}
 We define $\RR[\bY]_{\le D}^{\cT}$ to be the set of all polynomials of the form
\begin{equation}\label{eq:p-defn}
p(\bY) = \sum_{\Tree \in \cT_{\le D}} \hat{p}_{\Tree} \Poly_{\Tree}(\bY)
\end{equation}
for coefficients $\hat{p}_{\Tree} \in \RR$.
\end{definition}
 Note that any $p \in \RR[\bY]_{\le D}^{\cT}$ 
is a polynomial of degree at most $D$, that is 
\begin{align}\label{eq:poly-containment}
\RR[\bY]_{\le D}^{\cT} \subseteq \RR[\bY]_{\le D}\, .
\end{align}
We are now in position to state our result about reduction to tree-structured polynomials.
\begin{proposition}\label{prop:tree-vs-poly}
Assume $\pi_{\Theta}$ to have finite moments of all orders.
Let $\psi:\reals\to \reals$ be a measurable function, and assume all moments
of $\psi(\theta_1)$  to be finite.
For any fixed $\pi_{\Theta}, \psi, D$ there exists a fixed ($n$-independent) choice of coefficients
 $(\hat{p}_{\Tree})_{\Tree \in \cT_{\le D}}$ 
such that the associated polynomial $p = p_n \in \RR[\bY]^\cT_{\le D}$ defined by~\eqref{eq:p-defn} satisfies
\begin{align}\label{eq:tree-vs-poly}
\lim_{n \to \infty} \EE[(p(\bY) - \psi(\theta_1))^2] = \lim_{n \to \infty} \inf_{q \in \RR[\bY]_{\le D}} \EE[(q(\bY) - \psi(\theta_1))^2] \, .
\end{align}
 (In particular, the above limits exist.)
\end{proposition}

\subsection{Proof of Proposition \ref{prop:tree-vs-poly}}
\label{sec:KeyPropo}

The direction ``$\ge$'' in Eq.~\eqref{eq:tree-vs-poly} is immediate from Eq.~\eqref{eq:poly-containment}, so it remains to prove ``$\le$.''

For the proof, we will consider a slightly different model
whereby we observe $\tbY= \btheta\btheta^{\sT}/\sqrt{n}+\tbZ$, with $\tbZ=\tbZ^{\sT}$
and $(\tZ_{ij})_{1 \le i\le j\le n}\sim_{iid}\cN(0,1)$.  In other words, we use Gaussians of variance $1$ instead of $2$ on the diagonal. The original model is related to this one by $\bY = \tbY +\bW$
where $\bW$ is a diagonal matrix with $(W_{ii})_{1 \le i\le n}\sim_{iid}\cN(0,1)$ independent of $\btheta,\tbZ$. It is easy to see that it is sufficient to prove Proposition~\ref{prop:tree-vs-poly}
under this modified model. Indeed, the left-hand side of Eq.~\eqref{eq:tree-vs-poly}
does not depend on the distribution of the diagonal entries of $\bY$ (since tree-structured polynomials do not depend on the diagonal entries). For the right-hand side, if $q(\bY)$ is an arbitrary degree-$D$ estimator
then $\tilde{q}(\tbY) = \E_{\bW}q(\tbY+\bW)$ also has degree $D$ and satisfies
$\EE[q(\bY) \cdot \psi(\theta_1)]=\EE[\tilde{q}(\tbY) \cdot \psi(\theta_1)]$
and by Jensen's inequality, $\EE[q(\bY)^2]\ge \EE[\tilde{q}(\tbY)^2]$, whence we conclude 
\begin{equation}
\inf_{\tilde{q} \in \RR[\tbY]_{\le D}} \EE[(\tilde{q}(\tbY) - \psi(\theta_1))^2] \le \inf_{q \in \RR[\bY]_{\le D}} \EE[(q(\bY) - \psi(\theta_1))^2]\, .
\end{equation}
This proves the claim. In what follows, we will drop the tilde from $\tbY$, $\tbZ$ and assume the new normalization.

It is useful to generalize the setting introduced previously.
Let $\cG_{\le D}$ denote the set of all rooted (multi-)graphs, up to root-preserving isomorphism, with at most $D$ total edges, with the additional constraint that no vertices are isolated except possibly the root. Self-loops and multi-edges are allowed, and edges are counted with their multiplicity. For instance,  a triple-edge contributes 3 towards the edge count $D$. The graph need not be connected. For $\Graph \in \cG_{\le D}$ we write $V(\Graph)$ for the set of vertices and $E(\Graph)$ for the multiset of edges.

We use $\circ$ for  the root of a graph $\Graph \in \cG_{\le D}$, and 
define labelings (rooted at 1) exactly as for trees. Instead of non-reversing labelings, it will be convenient to work with \emph{embeddings}, that is, labelings that are injective (every pair of distinct vertices $u,v \in V(\Graph)$ has $\phi(u) \ne \phi(v)$). We denote the set of embeddings of $\Graph$ by $\emb(\Graph)$.

A labeling $\phi$ of $\Graph\in \cG_{\le D}$ induces a multi-graph whose vertices are elements of $[n]$.
This is the graph with vertex set $
\{\phi(v) \, : \; v\in V(\Graph)\}$  and edge set $\{(\phi(u),\phi(v)) \, : \; (u,v)\in E(\Graph)\}$. If $\phi(u)= \phi(u')$ and $\phi(v)=\phi(v')$
for $(u,v)$, $(u',v')\in E(\Graph)$ distinct edges in the multigraph $\Graph$, then
$(\phi(u),\phi(v))$ and $(\phi(u'),\phi(v'))$ are considered distinct edges in the induced multi-graph.


We call this induced graph the \emph{image} of $\phi$ and write $\balpha = \sh(\phi;\Graph)$
whenever $\balpha$ is the image of $\phi$.
We will treat $\balpha$ as an element of $\NN^\Pairs$ where $\Pairs=\{(i,j) \,:\, \, 1 \le i \le j \le n\}$, namely $\balpha=(\alpha_{ij})_{1 \le i\le j\le n}$ where $\alpha_{ij}\in\NN=
 \{0,1,2,\ldots\}$ counts the multiplicity of edge $(i,j)$ in $\balpha$. Formally, $\alpha_{ij}:=|\{(u,v)\in E(\Graph) \, : \, \phi(\{u,v\}) = \{i,j\}\}|$.

For $k \in \NN$, let $h_k : \RR \to \RR$ denote the $k$-th orthonormal Hermite polynomial. Recall that these are defined (uniquely, up to sign) by the following two conditions: $(i)$ $h_k$ is a degree-$k$ polynomial; $(ii)$~$\E[h_j(Z) h_k(Z)] = \One_{j=k}$ when $Z \sim \normal(0,1)$. We refer to~\cite{orthog-poly} for background.
  
For $\balpha \in \NN^\Pairs$, define the multivariate Hermite polynomial
\begin{align}
h_\balpha(\bY) = \prod_{1 \le i \le j \le n} h_{\alpha_{ij}}(Y_{ij})\,. 
\end{align}
These polynomials are orthonormal: 
$\E[h_\balpha(\bZ) h_\bbeta(\bZ)] = \One_{\balpha=\bbeta}$ when $(Z_{ij})_{i \le j} \sim_{iid} \normal(0,1)$. 
Viewing $\balpha$ as a graph, let $C(\alpha)$ denote the set of non-empty 
(i.e., containing at least one edge) connected components of $\balpha$. As above, each 
$\bgamma \in C(\balpha)$ is an element of $\NN^\Pairs$ where $\gamma_{ij}$ denotes the 
multiplicity of edge $(i,j)$. It will be important to ``center'' each component in the following 
sense. Define
\begin{equation}\label{eq:H_alpha}
\PolH_\balpha(\bY) = \prod_{\bgamma \in C(\balpha)} (h_\bgamma(\bY) - \E h_\bgamma(\bY))\,,
\end{equation}
where (in the case $\balpha = \bzero$) the empty product is equal to 1 by convention.
 Here and throughout, expectation is over $\bY$ distributed according to the rank-one estimation model  as defined 
 in Eq.~\eqref{eq:ModelDef}.
  For $\Graph \in \cG_{\le D}$, define
\begin{equation}\label{eq:H_A}
\PolH_{\Graph}(\bY) = \frac{1}{\sqrt{|\emb(\Graph)|}} \sum_{\phi \in \emb(\Graph)} 
\PolH_{\sh(\phi;\Graph)}(\bY)\,.
\end{equation}
Define the \emph{symmetric subspace} $\RR[\bY]^\sym_{\le D} \subseteq \RR[\bY]_{\le D}$ 
as
\begin{align}
\RR[\bY]^\sym_{\le D} :=\Big\{
f(\bY)=\!\!\sum_{\Graph\in\cG_{\le D}}\! a_{\Graph} \PolH_{\Graph}(\bY) \; : \; a_{\Graph}\in\RR
\;\; \forall \Graph\in\cG_{\le D}\Big\}\, .
\end{align}
In words, $\RR[\bY]^\sym_{\le D}$ is the $\RR$-span of $(\PolH_\Graph)_{\Graph \in \cG_{\le D}}$.
It is also easy to see that $\RR[\bY]^\sym_{\le D}$ is the linear subspace of 
$\RR[\bY]_{\le D}$  which is invariant under permutations of rows/columns
$\{2,\dots,n\}$ of $\bY$.

Note that the task of estimating $\psi(\theta_1)$ under the model \eqref{eq:ModelDef} is invariant under permutations of $\{2,\dots, n\}$.
As a consequence of the Hunt--Stein theorem \cite{eaton2021charles}, the optimal estimator of $\psi(\theta_1)$
must be equivariant under permutations of $\{2,\dots,n\}$. 
The following  lemma shows that the same is true if we restrict our attention to Low-Deg estimators.
Namely, instead of the infimum over all degree-$D$ polynomials
 in~\eqref{eq:tree-vs-poly}, 
it suffices to study only the symmetric subspace.
\begin{lemma}\label{lem:q-sym}
Under the assumptions of Proposition \ref{prop:tree-vs-poly}, for any $n$,
\begin{equation}\label{eq:q-sym}
\inf_{q \in \RR[\bY]_{\le D}} \EE[(q(\bY) - \psi(\theta_1))^2] = 
\inf_{q \in \RR[\bY]^\sym_{\le D}} \EE[(q(\bY) - \psi(\theta_1))^2] \, .
\end{equation}
\end{lemma}
Our next step will be to write down an explicit formula (given in Lemma~\ref{lem:q-formula}
below) for the right-hand side of~\eqref{eq:q-sym}. Define the vector 
$\bc_n = (c_{n,A})_{A \in \cG_{\le D}}$ and matrix $\bM_{n} = (M_{n,AB})_{A,B \in \cG_{\le D}}$ by
\begin{align}
c_{n,A} = \E[\PolH_A(\bY) \cdot \psi(\theta_1)] \, ,\;\;\;\;\;
 M_{n,AB} = \E[\PolH_A(\bY) \PolH_B(\bY)] \,. \label{eq:cMDef}
 \end{align}
Note that both $\bc_n$ and $\bM_n$ have constant dimension (depending on $D$ but not $n$), but 
their entries depend on $n$. Also note that $\bM_n$ is a Gram matrix and therefore 
positive semidefinite.
 We will show that $\bM_n$ is strictly positive definite (and thus invertible), and in fact
  strongly so in the sense that its minimum eigenvalue is lower bounded by a positive constant 
  independent of $n$.
  (Here and throughout, asymptotic notation such as $O(\,\cdot\,)$ and $\Omega(\,\cdot\,)$ 
may hide factors depending on $\pi_{\Theta},\psi,D$.)
\begin{lemma}\label{lem:M-pd}
Under the assumptions of Proposition \ref{prop:tree-vs-poly},
\[ \lambda_{\mathrm{min}}(\bM_n) = \Omega(1) \,. \]
\end{lemma}
\noindent We can now obtain an explicit formula for the optimal estimation error 
in terms of $\bM_n, \bc_n$.
\begin{lemma}\label{lem:q-formula}
Define the vector $\bc_n$ and matrix $\bM_n$ as per Eq.~\eqref{eq:cMDef}. Then, under
the assumptions of Proposition~\ref{prop:tree-vs-poly}, for any $n$,
\begin{align}
\inf_{q \in \RR[\bY]^\sym_{\le D}} \EE[(q(\bY) - \psi(\theta_1))^2] = \EE[\psi(\theta_1)^2] - \<\bc_n, \bM_n^{-1} \bc_n\>\,. 
\end{align}
Furthermore, the infimum is attained and the maximizer $q^*$, which is unique, takes the form
\begin{align}
q^*(\bY) = \sum_{A \in \cG_{\le D}} \hat{q}_A \PolH_A(\bY) \, ,\;\;\;\;\; \hbq =
(\hat{q}_A)_{A \in \cG_{\le D}}
= \bM_n^{-1} \bc_n\, .
\end{align}
\end{lemma}

Determining the asymptotic error achieved by Low-Deg estimators requires understanding the asymptotic behavior of $\bc_n$ and $\bM_n$. This is achieved in the following lemmas. Notably, $\bM_n$ is nearly block diagonal. (For this block diagonal structure to appear, it is crucial that we center each connected component in~\eqref{eq:H_alpha}.)
\begin{lemma}\label{lem:c-lim}
For each $A \in \cG_{\le D}$ we have the following.
\begin{itemize}
\item There is a constant $c_{\infty,A} \in \RR$ (depending on $\pi_{\Theta}, \psi, A$) such that
 $c_{n,A} = c_{\infty,A} + O(n^{-1/2})$.
\item If $A \notin \cT_{\le D}$ then $c_{\infty,A} = 0$.
\end{itemize}
\end{lemma}
\begin{lemma}\label{lem:M-lim}
For each $A,B \in \cG_{\le D}$ we have the following.
\begin{itemize}
\item There is a constant $M_{\infty,AB} \in \RR$ (depending on $\pi_{\Theta}, A, B$) such that 
$M_{n,AB} = M_{\infty,AB} + O(n^{-1/2})$.
\item If $A \in \cT_{\le D}$ and $B \in \cG_{\le D} \setminus \cT_{\le D}$ then $M_{\infty,AB} = 0$.
\end{itemize}
\end{lemma} 
Write $\bc_n, \bM_n$ in block form, where the first block is indexed by rooted trees 
$\cT_{\le D}$ and the second by other graphs $\cG_{\le D} \setminus \cT_{\le D}$:
\begin{equation}\label{eq:block-form}
\bc_n = \left[\begin{array}{c} \bd_n \\ \be_n \end{array}\right] \qquad\qquad
 \bM_n = \left[\begin{array}{cc} \bP_n & \bR_n \\ \bR_n^\sT & \bQ_n \end{array}\right].
\end{equation}
Let $\bd_{\infty}, \be_{\infty}, \bP_{\infty}, \bQ_{\infty}, \bR_{\infty}$ denote 
the corresponding limiting vectors/matrices from 
Lemmas~\ref{lem:c-lim} and~\ref{lem:M-lim}, and note that $\be_{\infty} = \bzero$ and $\bR_{\infty} = \bzero$.

We now work towards constructing the (sequence of) tree-structured polynomials $p = p_n$ that verify Eq.~\eqref{eq:tree-vs-poly}. We 
begin by defining a related polynomial $r = r_n$, which is tree-structured in the basis $\{\PolH_{\Tree}\}$ instead of $\{\Poly_{\Tree}\}$:
\begin{equation}\label{eq:r-def}
r(\bY) = \sum_{T \in \cT_{\le D}} \hat{r}_T \PolH_T(\bY) \qquad\text{where}\qquad \hbr = 
\bP_{\infty}^{-1} \bd_{\infty}\, .
\end{equation}
To see that this definition is well-posed, note that since $\bM_{\infty}$ is strictly positive definite by Lemma~\ref{lem:M-pd}, $\bP_{\infty}$ is also strictly positive definite and thus invertible.
For intuition, note the similarity between $\hbr$ and the optimizer $\hbq$ in 
Lemma~\ref{lem:q-formula}.

The next lemma characterizes the asymptotic estimation error achieved by the polynomial $r(\bY)$.
\begin{lemma}\label{lem:r-success}
We have
\begin{align}
\lim_{n \to \infty} \EE[r(\bY) \cdot \psi(\theta_1)] = \lim_{n \to \infty} \EE[r(\bY)^2] = 
\<\bd_{\infty}, \bP_{\infty}^{-1} \bd_{\infty}\> = \<\bc_{\infty}, \bM_{\infty}^{-1} \bc_{\infty}\>\, ,
\end{align}
and in particular,
\begin{align}
\lim_{n \to \infty} \EE[(r(\bY)-\psi(\theta_1))^2] = \EE[\psi(\theta_1)^2] - \langle \bc_\infty, \bM_\infty^{-1} \bc_\infty \rangle \, .
\end{align}
\end{lemma}
\noindent The crucial equality $\<\bd_{\infty}, \bP_{\infty}^{-1} \bd_{\infty}\> = \<\bc_{\infty}, \bM_{\infty}^{-1} 
\bc_{\infty}\>$ above
is an immediate consequence of the structure of $\bc_{\infty}, \bM_{\infty}$
established in Lemmas~\ref{lem:c-lim} and \ref{lem:M-lim}, namely $\be_{\infty} = \bzero$ and $\bR_{\infty} = \bzero$.

As a direct consequence of Lemma \ref{lem:r-success}, we can now prove the following analogue of Proposition~\ref{prop:tree-vs-poly} where  the basis $\{\PolH_{\Tree}\}$ is used instead of $\{\Poly_{\Tree}\}$.
which was left out by Lemma \ref{lem:r-success}.
\begin{proposition}\label{prop:tree-vs-poly-emb}
Assume $\pi_{\Theta}$ to have finite moments of all orders.
Let $\psi:\reals\to \reals$ be a measurable function, and assume all moments
of $\psi(\theta_1)$ to be finite. 
For any fixed $\pi_{\Theta}, \psi, D$ there exists a fixed ($n$-independent) choice of coefficients $(\hat{r}_{\Tree})_{\Tree \in \cT_{\le D}}$ such that the associated polynomial $r = r_n$ defined by 
 $r(\bY) = \sum_{\Tree \in \cT_{\le D}} \hat{r}_\Tree \PolH_{\Tree}(\bY)$ satisfies
\begin{equation}\label{eq:tree-vs-poly-emb}
\lim_{n \to \infty} \EE[(r(\bY) - \psi(\theta_1))^2] = \lim_{n \to \infty} \inf_{q \in \RR[\bY]_{\le D}} \EE[(q(\bY) - \psi(\theta_1))^2] \, .
\end{equation}
(In particular, the above limits exist.)
\end{proposition}
\begin{proof}[Proof of Proposition~\ref{prop:tree-vs-poly-emb}]
It follows immediately from Lemmas~\ref{lem:M-pd},~\ref{lem:c-lim}, and~\ref{lem:M-lim} that
\[ \lim_{n \to \infty} \<\bc_n, \bM_n^{-1} \bc_n\> = \<\bc_{\infty}, \bM_{\infty}^{-1} \bc_\infty\> \, . \]
(See Lemma \ref{lem:limit-exists} in the appendix for a formal proof.)
Combining this fact with Lemmas~\ref{lem:q-sym} and \ref{lem:q-formula},
the limit on the right-hand side of~\eqref{eq:tree-vs-poly-emb} exists and is equal to 
$\EE[\psi(\theta_1)^2] - \<\bc_{\infty}, \bM_{\infty}^{-1} \bc_{\infty}\>$.


Defining $\hbr$ as in~\eqref{eq:r-def}, we have by Lemma~\ref{lem:r-success} that the limit on the left-hand side 
of~\eqref{eq:tree-vs-poly-emb} is also equal to $\EE[\psi(\theta_1)^2] - \<\bc_{\infty}, \bM_{\infty}^{-1} \bc_{\infty}\>$. This completes the proof.
\end{proof}

Proposition \ref{prop:tree-vs-poly-emb} differs from our goal, namely Proposition~\ref{prop:tree-vs-poly}, in that it offers an estimator that is a linear combination
of the polynomials $\{\PolH_\Tree\}_{\Tree\in\cT_{\le D}}$ instead of $\{\Poly_{\Tree}\}_{\Tree\in\cT_{\le D}}$. Recalling that we are restricting to trees and that the first two Hermite polynomials are simply $h_0(z) = 1$ and $h_1(z) = z$, the difference lies in the fact that $\PolH_{\Tree}(\bY)$ involves a sum over embeddings, while $\Poly_{\Tree}(\bY)$ involves a sum over the larger class of non-reversing labelings; also $\PolH_\Tree(\bY)$ is centered by its expectation, provided $\Tree$ is not edgeless (see Eq.~\eqref{eq:H_alpha} and note that a tree has only one connected component).

The next lemma allows us to, for each $A \in \cT_{\le D}$, rewrite $\PolH_A(\bY)$ as a linear combination
of the $\{\Poly_{\Tree}(\bY)\}_{\Tree \in \cT_{\le D}}$ with
negligible error.
\begin{lemma}\label{lem:change-basis}
For any fixed $A \in \cT_{\le D}$ there exist $n$-independent coefficients 
$(m_{AB})_{B \in \cT_{\le D}}$ such that
\begin{align}
 \PolH_A(\bY) = \sum_{B \in \cT_{\le D}} m_{AB} \Poly_B(\bY) + \PolErr_A(\bY) \, ,\;\;\;\;\;
 \EE[\PolErr_A(\bY)^2] = o_n(1) \, .
\end{align}
\end{lemma}

\noindent We can finally conclude the proof of Proposition~\ref{prop:tree-vs-poly}.
\begin{proof}[Proof of Proposition~\ref{prop:tree-vs-poly}]
As in the proof of Proposition~\ref{prop:tree-vs-poly-emb}, it suffices to define $\hbp$ so that the limit on the
 left-hand side of~\eqref{eq:tree-vs-poly} is equal to $\EE[\psi(\theta_1)^2] - \<\bc_{\infty}, \bM_{\infty}^{-1} \bc_{\infty}\>$. Defining $\hbr$ as in~\eqref{eq:r-def}, we can use Lemma~\ref{lem:change-basis} to expand
\begin{align*}
r(\bY) &= \sum_{A \in \cT_{\le D}} \hat{r}_A \PolH_A(\bY) \\
&= \sum_{A \in \cT_{\le D}} \hat{r}_A \left[\sum_{B \in \cT_{\le D}} m_{AB} \Poly_B(\bY) + 
\PolErr_A(\bY)\right] \\
&= \sum_{A,B \in \cT_{\le D}} \hat{r}_A m_{AB} \Poly_B(\bY) + \sum_{A \in \cT_{\le D}} \hat{r}_A 
\PolErr_A(\bY) \\
&=: p(\bY) + \Delta(\bY)\,.
\end{align*}
In other words, we define $\hat p_B = \sum_{A \in \cT_{\le D}} \hat{r}_A m_{AB}$. 
Since $|\cT_{\le D}| = O(1)$, $\hat{r}_A = O(1)$, and $\EE[\PolErr_A(\bY)^2] = o_n(1)$, we have
\[ \EE[\Delta(\bY)^2]^{1/2}\le \sum_{A \in \cT_{\le D}} |\hat{r}_A| \cdot \E\left[ 
\PolErr_A(\bY)^2\right]^{1/2} = o_n(1)\, . \]
Now compute
\begin{align}
 \EE[p(\bY) \cdot \psi(\theta_1)] &=  \EE[r(\bY) \cdot \psi(\theta_1)] - 
 \EE[\Delta(\bY) \cdot  \psi(\theta_1)] \nonumber\\
 &= \<\bc_{\infty},\bM_{\infty}^{-1} \bc_{\infty}\> +o_n(1)\, , \label{eq:p-conclude-1}
 \end{align}
where the last step follows from  Lemma~\ref{lem:r-success} together with
the remark that $\EE[\psi(\theta_1)^2] = O(1)$, so $|\mathbb{E}[\Delta(\bY) \cdot  \psi(\theta_1)]| 
\le \EE[\Delta(\bY)^2]^{1/2} \EE[\psi(\theta_1)^2]^{1/2} = o_n(1)$. 

Similarly, since $\EE[\PolH_A(\bY)^2] = O(1)$ by Lemma~\ref{lem:M-lim}, and $\hat{r}_A = O(1)$, we have
\[ \EE[r(\bY)^2]^{1/2} \le \sum_{A \in \cT_{\le D}} |\hat{r}_A| \cdot \E\left[ 
\PolH_A(\bY)^2\right]^{1/2} = O(1)\, , \]
and therefore
\begin{align}
 \EE[p(\bY)^2]&=\EE[(r(\bY)-\Delta(\bY))^2] \nonumber\\
& = \EE[r(\bY)^2] - 2 \EE[r(\bY) \Delta(\bY)] + \EE[\Delta(\bY)^2]\nonumber\\
& = \<\bc_{\infty}, \bM_{\infty}^{-1} \bc_{\infty}\>
 + o_n(1)\,. \label{eq:p-conclude-2}
 \end{align} 
Combining Eqs.~\eqref{eq:p-conclude-1} and~\eqref{eq:p-conclude-2} we now conclude
\[ \lim_{n \to \infty} \EE[(p(\bY) - \psi(\theta_1))^2] = \EE[\psi(\theta_1)^2] - \<\bc_{\infty}, \bM_{\infty}^{-1} \bc_{\infty}\> \, , \]
completing the proof.
\end{proof}

\subsection{From tree-structured polynomials to AMP}
\label{sec:TreeToAMP}

In this section we will show that the Bayes-AMP algorithm achieves the same 
asymptotic accuracy as the best tree-structured constant-degree
polynomial $p \in \RR[\bY]_{\le D}^{\cT}$, matching
Eq.~\eqref{eq:tree-vs-poly}. This will complete the proof of the lower bound in
Theorem \ref{thm:Main}.

The key step is to construct a message passing (MP) algorithm that evaluates any given
tree-structured polynomial. We recall the definition of a class of MP algorithms
introduced in~\cite{bayati2015universality} (we make some simplifications with 
respect to the original setting). 
After $t$ iterations, the
 state of an algorithm in this class
is an array of \emph{messages} $(\bmm_{i\to j}^t)_{i,j\in [n]}$ indexed by ordered pairs 
in $[n]$. Each message is a vector $\bmm_{i\to j}^t\in\reals^{\dim}$, where $\dim$ is a 
fixed integer (independent of $n$ and $t$). Here, we use the arrow to emphasize ordering, and entries $i\to i$ are
set to $0$ by convention. 
Updates make use of functions $F_t:\reals^{\dim}\to\reals^{\dim}$ according to
\begin{align}
\bmm_{i\to j}^{t+1}&= \frac{1}{\sqrt{n}}\sum_{k\in [n]\setminus \{i,j\}}Y_{ik}F_t(\bmm^t_{k\to i})\, , 
\;\;\;\;\; \, \;\; \bmm_{i\to j}^0  =\bzero\;\;\forall i\neq j\, 
 .\label{eq:MP1}
\end{align}
We finally define vectors $\bmm_i^t,\hbmm_i^t\in\reals^{\dim}$ indexed by $i\in[n]$:
\begin{align}
\bmm_i^{t+1} &= \frac{1}{\sqrt{n}}\sum_{k\in [n]\setminus \{i\}}Y_{ik}F_t(\bmm^t_{k\to i})\, , \label{eq:MP2a}\\
\hbmm_{i}^{t+1}&= F_t(\bmm^{t+1}_{i})\, .\label{eq:MP2b}
\end{align}

We claim that this class of recursions can be used to evaluate the tree-structured polynomials
$\{\Poly_{\Tree}\}_{\Tree\in\cT_{\le D}}$, for any fixed $D$.
To see this, we let $\dim=|\cT_{\le D}|$ so that we can index the entries 
of $\bmm_{i\to j}^t$ or $\bmm_{i}^t, \hbmm_{i}^t$ by elements of $\cT_{\le D}$.
For instance, we will have:
\begin{align}
\bmm^t_{i\to j}= \big(\mm_{i\to j}^t(\Tree_{(1)}),\mm_{i\to j}^t(\Tree_{(2)}),\dots,
\mm_{i\to j}^t(\Tree_{(\dim)})\big)\, ,
\end{align}
where $\cT_{\le D} = (\Tree_{(1)},\Tree_{(2)},\dots,\Tree_{(\dim)})$ is an enumeration of $\cT_{\le D}$.

Given $\Tree\in \cT_{\le D}$, we define $\Tree_+$ to be the graph with vertex
set $V(\Tree)\cup\{v_+\}$ (where $v_+$ is not an element of $V(\Tree)$), 
edge set $E(\Tree)\cup\{(v_0,v_+)\}$ where $v_0$ was the root of $\Tree$.
We set the root of $\Tree_+$ to be $\circ=v_+$. In words, $\Tree_+$ is obtained by attaching an edge
to the root of $\Tree$, and moving the root to the other endpoint of the new edge.

If the root of $\Tree\in \cT_{\le D}$ has degree $k$, we 
let $\Tree_1,\dots, \Tree_k$ denote the connected subgraphs of $\Tree$
that are obtained by removing the root and the edges incident to the root.
Notice that each $\Tree_{j}$ is a tree which we root at the unique vertex $v_j$
such that $(\circ,v_j)\in E(\Tree)$.
 We write
$\cD(\Tree)= \{\Tree_1,\Tree_2,\dots,\Tree_k\}$. We then define the
special mapping $F^*:\reals^{|\cT_{\le D}|}\to \reals^{|\cT_{\le D}|}$,
by letting, for any $\Tree\in \cT_{\le D}$,
\begin{align}
F^*(\bmm)(\Tree) := \prod_{\Tree'\in \cD(\Tree)}\mm(\Tree')\, .\label{eq:SpecialUpdate}
\end{align}
In order to clarify the connection between this MP algorithm and tree-structured polynomials,
we define the following modification of the tree-structured polynomials 
of Eq.~\eqref{eq:F_A}. 
For each $\Tree \in \cT_{\le D}$, and each pair of distinct indices
$i,j\in[n]$, we define
\begin{equation}\label{eq:F_directed}
\Poly_{\Tree,i\to j}(\bY) =  \frac{1}{n^{|E(\Tree)|/2}}\sum_{\phi \in \nr(\Tree;i\to j)} \, 
\prod_{(u,v) \in E(\Tree)} Y_{\phi(u),\phi(v)}\,.
\end{equation}
Here $\nr(\Tree;i\to j)$ is the set of labelings i.e., 
maps $\phi: V(\Tree) \to [n]$ that are non-reversing 
(in the same sense as Definition \ref{def:NonReversing}) and such that
the following hold:
\begin{itemize}
\item $\phi(\circ)=i$. (The labeling is rooted at $i$, not at $1$.)
\item For any $v\in V(\Tree)$ such that $(\circ,v)\in E(\Tree)$, we have 
$\phi(v)\neq j$.
\end{itemize}
Notice also the different normalization with respect to Eq.~\eqref{eq:F_A}.
However, by counting the choices at each vertex we have $|\nr(\Tree;i\to j)| = (n-2)^{|E(\Tree)|}$,
$|\nr(\Tree)| = (n-1)(n-2)^{|E(\Tree)|-1}$
and therefore 
\begin{equation}
\frac{|\nr(\Tree;i\to j)|}{n^{|E(\Tree)|}} = 1+o_n(1)\, ,\;\;\;
\frac{|\nr(\Tree)|}{n^{|E(\Tree)|}} = 1+o_n(1)\, .
\end{equation}
We also define the radius of a rooted graph
$\Graph$, $\rad(\Graph):=\max\{\dist_{\Graph}(\circ,v):\; v\in V(\Graph)\}$.
\begin{proposition}\label{propo:TreetoAMP}
Let $\bmm^t_{i\to j}$, $\bmm_i^t$, $\hbmm_{i}^{t}$ be the iterates defined by Eqs.~\eqref{eq:MP1},
\eqref{eq:MP2a}, \eqref{eq:MP2b} with $F_t=F^*$ given by Eq.~\eqref{eq:SpecialUpdate}.
Then, for any $\Tree \in \cT_{\le D}$, and any $t>\rad(\Tree)$, 
recalling the definition of $T_+$ given above, we have the following 
(here the $o_n(1)$ terms are uniform in $\bY$):
\begin{align*}
\bmm^t_{i\to j}(\Tree) &= \Poly_{\Tree_+,i\to j}(\bY)\, ,\\
\bmm_{1}^{t}(\Tree) &= \frac{|\nr(\Tree_+)|}{n^{|E(\Tree_+)|/2}}\Poly_{\Tree_+}(\bY)
= (1+o_n(1))\cdot \Poly_{\Tree_+}(\bY)\, ,\\
\hbmm_{1}^{t}(\Tree) &= \frac{|\nr(\Tree)|}{n^{|E(\Tree)|/2}}\Poly_{\Tree}(\bY)
= (1+o_n(1))\cdot \Poly_{\Tree}(\bY)\, .
\end{align*}
\end{proposition}
\begin{proof}
The proof is straightforward, and amounts to check that the first claim holds 
by induction on $\rad(\Graph)$.
\end{proof}

We next show that, for a broad class of choices of the functions
$F_t$, the high-dimensional asymptotics of the algorithm defined by Eqs.~\eqref{eq:MP1}, \eqref{eq:MP2a}, \eqref{eq:MP2b} is determined by a generalization of the state evolution recursion of 
Theorem \ref{thm:SE}. 
\begin{lemma}\label{lemma:SEVec}
Assume that $\pi_{\Theta}$ has finite moments of all orders and, for each $t\ge 0$,
 $F_t:\reals^{\dim}\to \reals^{\dim}$ is a polynomial independent of $n$.
Define the sequence of vectors $\bmu_t\in\reals^{\dim}$ and positive semidefinite matrices
$\bSigma_t\in \reals^{\dim\times \dim}$ via the state evolution equations 
\eqref{eq:GeneralSE1}, \eqref{eq:GeneralSE2}.

Then for any polynomial $\psi:\reals^{\dim}\times \reals\to\reals$,
the following limits hold for $(\Theta,\bG_t)\sim\pi_{\Theta}\otimes \normal(0,\bSigma_t)$:
\begin{align}
&\lim_{n\to\infty}
\E[\psi(\bmm^t_{1\to 2},\theta_1)] = 
\lim_{n\to\infty}
\E[\psi(\bmm^t_1,\theta_1)]= 
\E\psi\big(\bmu_t\Theta+\bG_t,\Theta\big)\, , \label{eq:SEVec0-Convergence}\\
&\lim_{n\to\infty}
\E[\psi(\hbmm^t_1,\theta_1)]= 
\E\psi\big(F_t(\bmu_t\Theta+\bG_t),\Theta\big)\, .\label{eq:SEVec1-Convergence}
\end{align}
\end{lemma}
\noindent The proof of this lemma is based on results from \cite{bayati2015universality} 
and will be presented in Appendix \ref{app:LemmaSE}.

We are now in position to conclude the proof of the lower bound~\eqref{eq:LowerBound} on the optimal error of Low-Deg 
estimators in Theorem \ref{thm:Main}.
The quantity we want to lower bound takes the form
\begin{align*}
\MSE_{\le D} :=& \lim_{n\to\infty} \inf_{\hbtheta\in \LD(D;n)}\frac{1}{n}\E\big\{\|\hbtheta(\bY)-\btheta\|_2^2\big\} \\
=& \lim_{n\to\infty} \inf_{\hbtheta\in \LD(D;n)}\frac{1}{n}\sum_{i=1}^n \E\left\{(\htheta_i(\bY)-\theta_i)^2\right\} \\
=& \lim_{n\to\infty} \inf_{q \in \RR[\bY]_{\le D}} \E\left\{(q(\bY)-\theta_1)^2\right\}
\intertext{which by Proposition~\ref{prop:tree-vs-poly} takes the form}
=& \lim_{n\to\infty} 
\E\Big\{\Big(\theta_1-\sum_{T\in\cT_{\le D}} \hat{p}_T \Poly_{\Tree}(\bY)
\Big)^2\Big\}\, ,
\end{align*}
for some $n$-independent choice of coefficients $\hat{p}_T$.
%
%
On the other hand, by Proposition \ref{propo:TreetoAMP}, letting
 $\hbmm_{1}^{t}$ be the iterates defined by Eqs.~\eqref{eq:MP1},
\eqref{eq:MP2a}, \eqref{eq:MP2b}, with  $F_t=F^*$ given by Eq.~\eqref{eq:SpecialUpdate},
we have, for any $t>D$, 
\begin{align}
 \lim_{n\to\infty}\E\big[\theta_1 \cdot \Poly_{\Tree}(\bY)\big]&=\lim_{n\to\infty}\E\big[\theta_1 \cdot \hmm^t_1(\Tree)\big]\\
& \stackrel{(a)}{=}\E[\Theta F^*(\bmu_t\Theta+\bG_t)(\Tree)]=\bmu_{t+1}(\Tree)\, ,
\end{align}
where in step $(a)$ we used 
Lemma \ref{lemma:SEVec}. (Here $\bmu_t$, $\bSigma_t$ are defined recursively by Eqs.~\eqref{eq:GeneralSE1}, \eqref{eq:GeneralSE2}. Recall that $\bmu_t\in\reals^{|\cT_{\le D}|}$ can be indexed by elements $T$ of $\cT_{\le D}$.)

Analogously, we obtain
\begin{align}
\lim_{n\to\infty} \Big(\E\big[\Poly_{\Tree_1}(\bY)\Poly_{\Tree_2}(\bY)\big]\Big)_{\Tree_1,\Tree_2\in\cT_{\le D}} &=
 \E\big[F^*(\bmu_t\Theta+\bG_t)F^*(\bmu_t\Theta+\bG_t)^{\sT}\big]=\bSigma_{t+1}\, .
\end{align}
Hence we obtain that the optimal error achieved by Low-Deg estimators is given by
\begin{align}
\MSE_{\le D} &=\lim_{n\to\infty} 
\E\Big\{\Big(\theta_1-\sum_{T\in\cT_{\le D}} \hat{p}_T \Poly_{\Tree}(\bY)
\Big)^2\Big\}\\
& = \E[\Theta^2]-2\<\hbp,\bmu_{t+1}\> + \<\hbp,\bSigma_{t+1}\hbp\>\, .
\end{align}
%
However, by Theorem \ref{thm:SE} there exists an AMP algorithm of the form~\eqref{eq:FirstAMP} that achieves exactly the same error. Simply take $\dim=|\cT_{\le D}|$,
$F_t=F^*$ as defined in Eq.~\eqref{eq:SpecialUpdate}, and 
$g_t(\bx^t) = \<\hat{p}_T ,F^*(\bx^t)\>$.
The desired lower bound follows by applying the optimality result of Theorem~\ref{thm:BayesSE}.

\section*{Acknowledgments}

This project was initiated when the authors were visiting the Simons Institute for the Theory of 
Computing during the program on Computational Complexity of Statistical Inference in Fall 2021:
we are grateful to the Simons Institute for its support.

AM was supported by the NSF through award DMS-2031883, the Simons Foundation through
Award 814639 for the Collaboration on the Theoretical Foundations of Deep Learning, the NSF grant
CCF-2006489 and the ONR grant N00014-18-1-2729, and a grant from Eric and Wendy Schmidt
at the Institute for Advanced Studies. Part of this work was carried out while Andrea Montanari
was on partial leave from Stanford and a Chief Scientist at Ndata Inc dba Project N. The present
research is unrelated to AM’s activity while on leave.

Part of this work was done while ASW was with the Simons Institute for the 
Theory of Computing, supported by a Simons-Berkeley Research Fellowship. Part of 
this work was done while ASW was with the Algorithms and Randomness Center at Georgia 
Tech, supported by NSF grants CCF-2007443 and CCF-2106444.

\newpage

\appendix

\section{Proof of Theorem \ref{thm:Earlier}}
\label{app:IT-result}

Let $t_0:=\E[\Theta]$. We assume, without loss of generality, $t_0>0$. Let $\pi_{t}$ be the version of 
$\pi_{\Theta}$ centered at $t\ge 0$, namely the law of $\Theta-t_0+t$ when $\Theta\sim\pi_{\Theta}$
(in particular $\pi_{\Theta}=\pi_{t_0}$). We finally let $\btheta_t$ be a vector with i.i.d.\ coordinates
with distribution $\pi_t$, and (for $\bZ\sim\GOE(n)$)
\begin{align}
\bY_t = \frac{1}{\sqrt{n}}\btheta_t\btheta_t^{\sT}+\bZ\, .
\end{align}
We will take $\btheta_t=\btheta_0+t\bfone$.
The normalized mutual information between $\bY_t$ and $\btheta_t$ is given, after
simple manipulations, by
\begin{align*}
\phi_n(t) :=& \frac{1}{n}I(\bY_t;\btheta_t) \\
=& -\frac{1}{n}\E_{\btheta_t,\bZ}\log\left\{
\int \exp\Big(-\frac{1}{4n}\big\|\obtheta_t\obtheta_t^{\sT}-\btheta_t\btheta_t^{\sT}\big\|^2_F
+\frac{1}{2n^{1/2}}\<\obtheta_t,\bZ\obtheta_t\>\Big)\,\pi_t^{\otimes n}(\de\obtheta_t)\right\}\, .
\end{align*}
By writing $\btheta_t = \btheta_0+t\bfone$, $\obtheta_t = \obtheta_0+t\bfone$, where 
$\btheta_0\sim \pi_0^{\otimes n}$, we get 
\begin{align}
\phi_n(t) &:= -\frac{1}{n}\E_{\btheta_0,\bZ}\log\left\{e^{-\cH_n(\obtheta_0,\btheta_0;t)}\pi_0^{\otimes n}(\de\obtheta_0)\right\}\, ,
\label{eq:phinHn}
\end{align}
where
\begin{align*}
\frac{\partial\cH_n}{\partial t}(\obtheta_0,\btheta_0;t) = \frac{1}{n}
\Big\{\|\obtheta_t-\btheta_t\|^2\<\obtheta_t,\bfone\>-\<\obtheta_t-\btheta_t,\obtheta_t\>\<\obtheta_t-\btheta_t\,\bfone\>\Big\}-\frac{1}{n^{1/2}}\<\bZ,\bfone\obtheta_t^{\sT}\>\, .
\end{align*}
Let $\mu_{\bY_t}(\de\obtheta_t)$ be the Bayes posterior over $\obtheta_t$:
\begin{align}
\mu_{\bY_t}(\de\obtheta_t) \propto \exp\Big(-\frac{1}{4}\Big\|\bY_t-\frac{1}{n^{1/2}}\obtheta_t\obtheta_t^{\sT}\Big\|_F^2\Big)\,
\pi^{\otimes n}_t(\de\obtheta_t)\, .
\end{align}
The above implies
\begin{align}
\phi'_n(t) &= A_n(t) + B_n(t)\, ,\\
A_n(t) &:= \frac{1}{n^2}\E_{\btheta_t,\bZ}\mu_{\bY_t}\Big(\|\obtheta_t-\btheta_t\|^2\<\obtheta_t,\bfone\>-
\<\obtheta_t-\btheta_t,\obtheta_t\>\<\obtheta_t-\btheta_t\,\bfone\> \Big)\, ,\\
B_n(t) &:= -\frac{1}{n^{3/2}}\E_{\btheta_t,\bZ}\<\bZ,\mu_{\bY_t}(\bfone\obtheta_t^{\sT})\>\, .
\end{align}
Integrating $\bZ$ by parts, we get
\begin{align}
B_n(t) &:= -\frac{1}{n^{2}}\E_{\btheta_t,\bZ}\Big\{\mu_{\bY_t}
\big(\<\bfone,\obtheta_t\>\|\obtheta_t\|^2\big)-\<\mu_{\bY_t}(\obtheta_t),
\mu_{\bY_t}(\<\bfone,\obtheta_t\>\obtheta_t)\>\,\Big\}\, .
\end{align}
We finally note that $\btheta_t,\obtheta_t$ are jointly distributed according to 
$\sE F(\btheta_t,\obtheta_t) = \E\mu_{\bY_t}^{\otimes 2}\big(F(\btheta_t,\obtheta_t) \big)$.
Note that the pair $\btheta_t,\obtheta_t$ is exchangeable, with marginals
$\btheta_{t}\sim \pi_t^{\otimes n}$, $\obtheta_t\sim \pi_t^{\otimes n}$.
Writing the above expectations in terms of this measure, we get
\begin{align}
A_n(t) &= \frac{1}{n^2}\sE\Big(\|\obtheta_t-\btheta_t\|^2\<\obtheta_t,\bfone\>-
\<\obtheta_t-\btheta_t,\obtheta_t\>\<\obtheta_t-\btheta_t,\bfone\> \Big)\, ,\label{eq:Afin}\\
B_n(t) & = -\frac{1}{n^{2}}\sE\Big(\<\bfone,\obtheta_t\>\|\obtheta_t\|^2-
\<\bfone,\obtheta_t\>\<\btheta_t,\obtheta_t\>\,\Big)\, .\label{eq:Bfin}
\end{align}
By concentration inequalities for sums of independent random variables,
there exist constants $C_k$ such that
\begin{align}
&\sE\big(\|\btheta_t\|_2^{2k}\big)  \le C_k n^k\, ,\\
&\sE\big(  \vert\<\btheta_t,\bfone\>-nt \vert^k\big)\le C_kn^{k/2}\, ,\\
&\sE\big(\vert\|\btheta_t\|_2^2-\sE[\|\btheta_t\|_2^2]\vert^{k}\big)  \le C_k n^k\, .
\end{align}
The same inequalities obviously hold for $\obtheta_t$.

Using these bounds in Eqs.~\eqref{eq:Afin}, \eqref{eq:Bfin}, we get
\begin{align}
A_n(t) &= \frac{t}{n}\sE\big\{\|\obtheta_t-\btheta_t\|^2\big\}+O(n^{-1/2})\, ,\\
B_{n}(t) &= -\frac{t}{n}\sE\{\|\btheta_t\|^2_2\}+\frac{t}{n^2}
\E\big\{\big\|\hbtheta_t(\bY_t)\big\|^2_2\big\}+O(n^{-1/2})\, .
\end{align}
Summing these, we finally obtain
\begin{align}
\phi'_n(t) &=\frac{t}{n}\E\big\{\big\|\btheta_t -\hbtheta_t(\bY_t)\big\|_2^2\big\}+O(n^{-1/2})
\label{eq:PhiPrimeMSE}\\
&=: t\, \MSE(t;n)+O(n^{-1/2})\, ,\nonumber
\end{align}
where we note that the $O(n^{-1/2})$ term is uniform in $t$.

Now recalling the definition of Eq.~\eqref{eq:PsiDef}, we 
have, by \cite[Theorem 1]{lelarge2019fundamental},
\begin{equation}
\phi_{\infty}(t):= \lim_{n\to\infty} \phi_n(t) = \inf_{q\ge 0} \Psi(q;0,\pi_t)  \, .\label{eq:PhiPhiInfty}
\end{equation}
Note that 
\begin{align}
\Psi(q;0,\pi_t) = \Psi(q;t^2,\pi_0)+\frac{1}{2}t^2\Var(\Theta)+\frac{1}{4}t^4\, .
\end{align}
Differentiating with respect to $t$, we deduce that
\begin{enumerate}
\item $\phi_{\infty}(t)$ is differentiable at $t$ if and only if  $b\mapsto \Psi(q;b,\pi_t)$
is differentiable at $b=0$. Both are in turn equivalent to $q\mapsto \Psi(q;0,\pi_t)$
being uniquely minimized at $q_{\sBayes}(\pi_{\Theta})$.
\item If the last point holds, then (with $\Theta_t = \Theta-t_0+t\sim\pi_t$)
\begin{align}
\phi_{\infty}'(t) =t\cdot \big(\E[\Theta^2_t]-q_{\sBayes}(\pi_{\Theta})\big)\, .
\end{align}
\end{enumerate}
Comparing with Eq.~\eqref{eq:PhiPrimeMSE}, we are left with the task of proving that
$\lim_{n\to\infty}\phi_{n}'(t_0)=\phi_{\infty}'(t_0)$ when $\phi_{\infty}$ is differentiable at $t_0$.

Taking the second derivative of Eq.~\eqref{eq:phinHn} yields
\begin{align*}
\phi''_n(t) &= -\frac{1}{n}\E_{\btheta_0,\bZ}\Var_{\mu_{\bY_t}}\big(\partial_t\cH_n(\obtheta_0,\btheta_0;t)\big)
+ \frac{1}{n}\E_{\btheta_0,\bZ}\mu_{\bY_t}\big(\partial^2_t\cH_n(\obtheta_0,\btheta_0;t)\big)\\
&=: -D_n(t)+E_n(t)\, ,\label{eq:PhiSec}
\end{align*} 
where $D_n(t)\ge 0$ by construction.
A direct calculation yields
\begin{align}
\frac{\partial^2\cH_n}{\partial t^2}(\obtheta_0,\btheta_0;t) = \frac{1}{2}\|\obtheta_t-\btheta_t\|^2
+\frac{1}{2n}\<\obtheta_t-\btheta_t,\bfone\>^2 -\frac{1}{2 n^{1/2}}\<\bfone,\bZ\bfone\>\, .
\end{align}
Hence proceeding as for the first derivative, we get
\begin{align}
E_n(t) &= \frac{1}{2n} \E_{\btheta_0,\bZ}\mu_{\bY_t}\big(\|\obtheta_t-\btheta_t\|^2\big)
+\frac{1}{2n} \E_{\btheta_0,\bZ}\mu_{\bY_t}\big(\<\obtheta_t-\btheta_t,\bfone\>^2\big)\\
& = \MSE(t;n)+ O(n^{-1/2})\, ,\label{eq:En}
\end{align}
where, once more, the $O(n^{-1/2})$ term is uniform in $t$.

Consider now the function $f_n(t)=-t^{-1/2}\, \phi_n(t)$, defined on an interval $[t_1,t_2]$,
with $0<t_1<t_0<t_2$ (with $n\in\naturals\cup \{\infty\}$). We conclude that
\begin{enumerate}
\item By Eq.~\eqref{eq:PhiPhiInfty}, $\lim_{n\to\infty}f_n(t) = f_{\infty}(t)$ for any
$t\in [0,1]$. (In fact the convergence is uniform since $\phi'_n(t)$ is uniformly bounded 
by Eq.~\eqref{eq:PhiPrimeMSE}, but we will not use this fact.)
\item We have $f'_n(t) = \phi_n(t)/(2t^{3/2})-\phi'_n(t)/t^{1/2}$. Hence, by the previous point,
$\lim_{n\to\infty}f'_n(t) =f'_{\infty}(t)$ at a point $t$ if and only if 
$\lim_{n\to\infty}\phi'_n(t) =\phi'_{\infty}(t)$.
\item We have  $f''_n(t) = -3\phi_n(t)/(4t^{5/2})+\phi'_n(t)/t^{3/2}-\phi''_n(t)/t^{1/2}$.
Substituting Eqs.~\eqref{eq:PhiPrimeMSE}, \eqref{eq:PhiSec}, \eqref{eq:En}, we obtain
\begin{align}
f''_n(t) = -\frac{3}{4t^{5/2}}\phi_n(t) +\frac{1}{t^{1/2}} D_n(t)\,.
\end{align}
In particular, the first term is uniformly bounded (by boundedness of the normalized
mutual information on $[t_1,t_2]$), and the second is
non-negative.
\end{enumerate}

Points 1 and 3 imply (by the lemma below) that $f'_n(t)\to f_{\infty}'(t)$ for any differentiability
point $t$, and hence by point 3, we conclude $\phi'_n(t)\to \phi_{\infty}'(t)$.
The proof is then completed by the following analysis fact.
\begin{lemma}
Let $\{f_n:n\ge 1\}$, $f_n:[t_1,t_2]\to \reals$ be a sequence of differentiable functions converging 
pointwise to $f_{\infty}$. Assume $f_n=g_n+h_n$ where $g_n,h_n$ are differentiable, $g_n$ is convex and the
$\{h_n'\}$ are equicontinuous on $[t_1,t_2]$ (i.e.\ there exists a non-decreasing function $\delta(\eps)\downarrow 0$
such that $\vert t-t'\vert\le \eps$ implies $\vert h_n'(t)-h'_n(t')\vert\le \delta(\eps)$).
If $f_{\infty}$ is differentiable at $t_0$, then $\lim_{n\to\infty}f'_n(t_0) = f'_{\infty}(t_0)$.
\end{lemma}
\begin{proof}
By convexity we have, for $\eps>0$
\begin{align}
\frac{1}{\eps}\big[g_n(t_0+\eps)-g_n(t_0)\big]\ge g'_n(t_0)\, .
\end{align}
Further, by the intermediate value theorem and equicontinuity of $h'_n$,
\begin{align}
\frac{1}{\eps}\big[h_n(t_0+\eps)-h_n(t_0)\big]\ge h'_n(t_n(\eps))
\ge h'_n(t_0)-\delta(\eps)\, .
\end{align}
Summing the last two displays, we get
\begin{align}
f_n'(t_0)\le\frac{1}{\eps}\big[f_n(t_0+\eps)-f_n(t_0)\big]+\delta(\eps)\, .
\end{align}
Taking the limit $n\to\infty$ and using the convergence of $f_n$, we get
\begin{align}
\limsup_{n\to\infty}f_n'(t_0)\le\frac{1}{\eps}\big[f_{\infty}(t_0+\eps)-f_\infty(t_0)\big]+\delta(\eps)\, .
\end{align}
Finally taking the limit $\eps\downarrow 0$, and using the differentiability of $f_{\infty}$
at $t_0$: 
\begin{align}
\limsup_{n\to\infty}f_n'(t_0)\le f'_{\infty}(t_0)\, .
\end{align}
\end{proof}

\section{Lemma for the upper bound in Theorem \ref{thm:Main}}
\label{app:UpperBound}

\begin{lemma}\label{lemma:BayesFact}
Let $f^{\sBayes}(x;\pi_{\Theta},q):=\E[\Theta|q\Theta+\sqrt{q} G = x]$ for $(\Theta,G)\sim\pi_{\Theta}\otimes\normal(0,1)$.
Then, for any $q>0$  there exists a constant $C=C(q,\pi_{\Theta})$ such that, for all $x$
\begin{align}
|f^{\sBayes}(x;\pi_{\Theta},q)|\le C(1+|x|)\, .
\end{align}
\end{lemma}
\begin{proof}
If $\pi^{a}_{\Theta}$ is the law of $\Theta+a$, then we have
$f^{\sBayes}(x;\pi_{\Theta},q) = f^{\sBayes}(x+qa;\pi^a_{\Theta},q)$.
Hence, if the claim holds for $\pi^a_{\Theta}$, it holds for $\pi_{\Theta}$
as well. Hence, without loss of generality, we can assume $\E[\Theta]=0$.

Note that $f^{\sBayes}(-x;\pi_{\Theta},q) = -f^{\sBayes}(x;\pi_{-\Theta},q)$
where $\pi_{-\Theta}$ is the law of $-\Theta$. 
Therefore, without loss of generality we can assume $x>0$.
A simple calculation (`Tweedie's formula') yields
\begin{align}
f^{\sBayes}(x;\pi_{\Theta},q) &= \frac{\de\phantom{x}}{\de x}\phi(x;\pi_{\Theta},q)\, ,\\
\phi(x;\pi_{\Theta},q) & :=
\log\Big\{ \int e^{\theta x-q\theta^2/2}\pi_{\Theta}(\de \theta)\Big\}\, .
\end{align}
Notice that $x\mapsto \phi(x;\pi_{\Theta},q)$ is convex, and therefore
\begin{align}\label{eq:BoundDeltaBayes}
&\Delta_q(-2x,-x)\le 
f^{\sBayes}(x;\pi_{\Theta},q) \le \Delta_q(x,2x)\, ,\\
&\Delta_q(x_1,x_2):=\frac{1}{x_2-x_1}\big\{\phi(x_2;\pi_{\Theta},q)-\phi(x_1;\pi_{\Theta},q)
\big\}\, .
\end{align}
We therefore proceed to lower and upper bound $\phi$.
First notice that
\begin{align*}
\phi(x;\pi_{\Theta},q) &= \frac{x^2}{2q}+\log\Big\{ \int e^{-(x-q\theta)^2/2q}\pi_{\Theta}(\de \theta)\Big\}\\
&\le \frac{x^2}{2q}\, ,
\end{align*}
where we used the fact that $\exp(-(x-q\theta)^2/2q)\le 1$. Further, by Jensen's inequality
\begin{align*}
\phi(x;\pi_{\Theta},q)\ge x\E[\Theta]-\frac{1}{2}\E(\Theta^2)q = -\frac{1}{2}\E(\Theta^2)q\, .
\end{align*}
Therefore we obtain, for $x>1$, 
\begin{align*}
f^{\sBayes}(x;\pi_{\Theta},q) \le \Delta_q(x,2x)
&\le \frac{x}{2q} + \frac{1}{2x}\E(\Theta^2)q\\
&\le \frac{x}{2q} + \frac{1}{2}\E(\Theta^2)q\, .
\end{align*}
Since $x\mapsto f^{\sBayes}(x;\pi_{\Theta},q)$ is continuous and therefore bounded on $[0,1]$,
this proves $f^{\sBayes}(x;\pi_{\Theta},q) \le C(1+|x|)$.

In order to prove $f^{\sBayes}(x;\pi_{\Theta},q) \ge -C(1+|x|)$, 
we use the lower bound in Eq.~\eqref{eq:BoundDeltaBayes} and proceed analogously.
\end{proof}

\section{Proof of Lemmas for Proposition~\ref{prop:tree-vs-poly}}
\label{sec:technical-app}

\subsection{Notation}
\label{sec:notation}

We use the convention $\NN = \{0,1,2,\ldots\}$. We often work with 
$\balpha \in \NN^{\Pairs}$ where $\Pairs=\{(i,j) \,:\, \, 1 \le i \le j \le n\}$
and $|\Pairs| = N := n(n+1)/2$. For $\balpha,\bbeta \in \NN^\Pairs$, we use the notation
\[ |\balpha| := \sum_{(i,j)\in\Pairs} \alpha_{ij}, 
 \qquad
 \balpha! := \prod_{(i,j)\in\Pairs} \alpha_{ij}!, \]
\[ \binom{\balpha}{\bbeta} := \prod_{(i,j)\in\Pairs} \binom{\alpha_{ij}}{\beta_{ij}},
  \qquad
  \bY^\balpha := \prod_{(i,j)\in\Pairs} Y_{ij}^{\alpha_{ij}}, \]
\[  (\balpha \symd \bbeta)_{ij} = |\alpha_{ij} - \beta_{ij}|, \qquad
  (\balpha \wedge \bbeta)_{ij} = \min\{\alpha_{ij},\beta_{ij}\}, \qquad
  (\balpha \setminus \bbeta)_{ij} := \max\{0,\alpha_{ij} - \beta_{ij}\}\, . \]
We use $\balpha \le \bbeta$ to mean $\alpha_{ij} \le \beta_{ij}$ for all $i,j$, and 
$\balpha \lneq \bbeta$ to mean $\alpha_{ij} \le \beta_{ij}$ for all $i,j$ with strict 
inequality for some $i,j$. Note that $\balpha \in \NN^\Pairs$ 
can be identified with a multigraph whose vertices are elements of $[n]$, where 
$\alpha_{ij}$ is the number of copies of edge $(i,j)$. We use
 $V(\balpha) \subseteq [n]$ to denote the set of vertices spanned by the edges of
  $\balpha$, with vertex 1 (the root) always included by convention.
   We use $C(\balpha)$ to denote the set of non-empty (i.e., containing at least one edge) 
   connected components of $\balpha$.

Asymptotic notation such as $O(\;\cdot\;)$ and $\Omega(\;\cdot\;)$ pertains to the limit 
$n \to \infty$ and may hide factors depending on $\pi_{\Theta}, \psi, D$. We use the symbol 
$\const$ to denote a constant (which may be positive, zero, or negative) depending on 
$\pi_{\Theta}, \psi, D$ (but not $n$) and e.g., $\const(A,B)$ to denote a constant that may
 additionally depend on $A,B$. These constants may change from line to line.


\subsection{Hermite polynomials}
\label{sec:hermite}

We will need the following well known property of Hermite polynomials (see e.g.\ page 254 of~\cite{special-functions}).

\begin{proposition}
For any $\mu, z \in \RR$ and $k \in \NN$,
\[ h_k(\mu+z) = \sum_{\ell=0}^k \sqrt{\frac{\ell!}{k!}} \binom{k}{\ell} \mu^{k-\ell} h_\ell(z) \, . \]
\end{proposition}

\noindent As a result, for any $\bgamma \in \NN^\Pairs$, and writing $\bX=\btheta\btheta^{\sT}/\sqrt{n}$
for the signal matrix, we have:
\begin{align*}
h_{\bgamma}(\bY) &= \prod_{(i,j)\in\Pairs} h_{\gamma_{ij}}(X_{ij} + Z_{ij}) \\
&= \prod_{(i,j)\in\Pairs} \sum_{\beta_{ij}=0}^{\gamma_{ij}} \sqrt{\frac{\beta_{ij}!}{\gamma_{ij}!}} \binom{\gamma_{ij}}{\beta_{ij}} X_{ij}^{\gamma_{ij}-\beta_{ij}} h_{\beta_{ij}}(Z_{ij}) \\
&= \sum_{\bzero \le \bbeta \le \bgamma} \prod_{(i, j)\in\Pairs} \sqrt{\frac{\beta_{ij}!}{\gamma_{ij}!}} \binom{\gamma_{ij}}{\beta_{ij}} X_{ij}^{\gamma_{ij}-\beta_{ij}} h_{\beta_{ij}}(Z_{ij}) \\
&= \sum_{\bzero \le \bbeta \le \bgamma} \sqrt{\frac{\bbeta!}{\bgamma!}} \binom{\bgamma}{\bbeta} 
\bX^{\bgamma-\bbeta} h_\bbeta(\bZ) \, .
\end{align*}
In particular, since $\EE_{\bZ}[h_\bbeta(\bZ)] = \One_{\bbeta = \bzero}$ 
(due to orthonormality of Hermite polynomials and the fact $h_0(z) = 1$),
\[ \EE h_\bgamma(\bY) = \frac{1}{\sqrt{\bgamma!}} \EE \bX^\bgamma \, . \]

\subsection{Proof of Lemma~\ref{lem:q-sym}: Symmetry}

As mentioned in the main text, this is a special case of the classical Hunt--Stein theorem
\cite{eaton2021charles}, the only difference being that we are restricting the class of estimators to degree-$D$ polynomials. We present a proof for completeness.

Let $S_{-1}$ denote the group of permutations on $[n]$ that fix $\{1\}$. This group acts 
on the space of $n \times n$ symmetric matrices by permuting both the rows and columns 
(by the same permutation). We also have an induced action of $S_{-1}$ on $\RR[\bY]_{\le D}$ given 
by $(\sigma \cdot q)(\bY) = q(\sigma^{-1} \cdot \bY)$. A polynomial $q$ can be symmetrized 
over $S_{-1}$ to produce $q^\sym$ as follows:
\[ q^\sym := \frac{1}{|S_{-1}|} \sum_{\sigma \in S_{-1}} \sigma \cdot q\,. \]

We claim that the symmetric subspace $\RR[\bY]_{\le D}^\sym = 
\mathrm{span}_\RR\{\PolH_A \,:\, \cG_{\le D}\}$ is precisely the image of the 
above symmetrization operation, that is, $\RR[\bY]_{\le D}^\sym = \{q^\sym \,:\, q \in 
\RR[\bY]_{\le D}\}$. This can be seen as follows. For the inclusion $\subseteq$, 
note that any $q \in \RR[\bY]_{\le D}^\sym$ satisfies $q = q^\sym$. For the reverse 
inclusion $\supseteq$, start with an arbitrary $q \in \RR[\bY]_{\le D}$. Any such $q$ 
admits an expansion as a linear combination of Hermite polynomials 
$(h_\balpha)_{|\balpha| \le D}$, and therefore admits an expansion in 
$(\PolH_\balpha)_{|\balpha| \le D}$ (cf.\ the definition~\eqref{eq:H_alpha}
which can be inverted recursively). Finally, note that $\PolH_\balpha^\sym$ is a 
scalar multiple of $\PolH_{\Graph}$ for a certain $\Graph \in \cG_{\le D}$. This 
allows $q^\sym$ to be written as a linear combination of $(\PolH_\Graph)_{\Graph \in \cG_{\le D}}$.

To complete the proof of Lemma~\ref{lem:q-sym}, it is sufficient to show that for any 
low degree estimator $q\in \RR[\bY]_{\le D}$, mean squared error does not increase under symmetrization:
\begin{equation}\label{eq:sym-improve}
\EE[(q^\sym(\bY) - \psi(\theta_1))^2] \le
\EE[(q(\bY) - \psi(\theta_1))^2]\,.
\end{equation}
First note that $\EE[q^\sym(\bY) \cdot \psi(\theta_1)] 
  = \EE[q(\bY) \cdot \psi(\theta_1)]$ because $(\bY,\theta_1)$ and $(\sigma \cdot \bY,\theta_1)$
have the same distribution for any fixed $\sigma \in S_{-1}$.
Next, using Jensen's inequality and the symmetry $\EE[q^2] = 
   \EE[(\sigma \cdot q)^2]$ for all $\sigma \in S_{-1}$,
 we have
 \begin{align*}
 \EE[q^\sym(\bY)^2] &=\EE\Big\{\Big(\frac{1}{|S_{-1}|} \sum_{\sigma \in S_{-1}}
 (\sigma \cdot q)(\bY)\Big)^2\Big\}\\
 &\le 
 \frac{1}{|S_{-1}|} \sum_{\sigma \in S_{-1}} \EE[(\sigma \cdot q)(\bY)^2] = \EE[q(\bY)^2]\,.
 \end{align*}
Now Eq.~\eqref{eq:sym-improve} follows by combining these claims.

\subsection{Proof of Lemma~\ref{lem:M-pd}: $\lambda_{\mathrm{min}}(\bM_n)$}

Consider an arbitrary polynomial 
$f = \sum_{\Graph \in \cG_{\le D}} b_{\Graph}(f) \PolH_{\Graph}$ with coefficients 
$\bb(f) =(b_\Graph(f))$ normalized so that $\|\bb(f)\|^2 := \sum_{\Graph \in \cG_{\le D}} b_\Graph(f)^2 = 1$. This induces an expansion $f = \sum_{\balpha \in \NN^\Pairs, |\balpha| \le D} \hat{f}_\balpha \PolH_\balpha$.
Since $\EE[f(\bY)^2] = \<\bb, \bM_n \bb\>$, it suffices to show 
$\EE[f(\bY)^2] = \Omega(1)$. 
Using Jensen's inequality and orthonormality of Hermite polynomials
(here, we use subscripts on expectation to denote which variable is being integrated over),
\begin{align*}
\EE[f(\bY)^2] = \EE_Z \EE_X [f(\bX+\bZ)^2] \ge \EE_Z[(\EE_X f(\bX+\bZ))^2] =: \|\hat{g}\|^2
\end{align*}
where $g(\bZ) := \EE_\bX f(\bX+\bZ)$ with Hermite expansion $g = \sum_{\balpha \in \NN^\Pairs, |\balpha| \le D} \hat{g}_\balpha \PolH_\balpha$.
It now suffices to show $\|\hat g\|^2 = \Omega(1)$. 
We will compute $\hat{g}$ explicitly in terms of $\hat{f}$. 
We have, using the Hermite facts from Section~\ref{sec:hermite},
\begin{align*}
g(\bZ) &= \EE_X f(\bX+\bZ) \\
&= \EE_X \sum_{\balpha\in \NN^\Pairs} \hat{f}_\balpha \PolH_\balpha(\bX+\bZ) \\
&= \sum_{\balpha\in \NN^\Pairs} \hat{f}_\balpha \EE_X \prod_{\gamma \in C(\balpha)} \left(h_\bgamma(\bX+\bZ) - \EE_Y h_\bgamma(\bY)\right) \\
&= \sum_{\balpha\in \NN^\Pairs} \hat{f}_\balpha \prod_{\bgamma \in C(\balpha)} \EE_X[h_\bgamma(\bX+\bZ) - \EE_Y h_\bgamma(\bY)] \qquad\text{using independence between components} \\
&= \sum_{\balpha\in \NN^\Pairs} \hat{f}_\balpha \prod_{\bgamma \in C(\balpha)} \EE_X\left[\sum_{0 \le \bbeta \le \bgamma} \sqrt{\frac{\bbeta!}{\bgamma!}}\binom{\bgamma}{\bbeta}X^{\bgamma-\bbeta}h_\bbeta(\bZ) - \frac{1}{\sqrt{\bgamma!}} \bX^\bgamma\right] \\
&= \sum_{\balpha\in \NN^\Pairs} \hat{f}_\balpha \prod_{\bgamma \in C(\balpha)} \sum_{0 \lneq \bbeta \le \bgamma} \sqrt{\frac{\bbeta!}{\bgamma!}}\binom{\bgamma}{\bbeta}\EE[\bX^{\bgamma-\bbeta}]h_\bbeta(\bZ) \\
&= \sum_{\balpha\in \NN^\Pairs} \hat{f}_\balpha \sum_{\bbeta \in \Gamma(\balpha)} \sqrt{\frac{\bbeta!}{\balpha!}}\binom{\balpha}{\bbeta}\EE[\bX^{\balpha-\bbeta}]h_\bbeta(\bZ)
\intertext{where $\Gamma(\balpha)$ is the set of $\bbeta \in \NN^\Pairs$ such that $0 \le \bbeta \le \balpha$, and $\bbeta$ includes at least one edge from every non-empty component of $\balpha$}
&= \sum_{\bbeta\in \NN^\Pairs} h_\bbeta(\bZ) \sum_{\balpha \,:\, \bbeta \in \Gamma(\balpha)} \hat{f}_\balpha \sqrt{\frac{\bbeta!}{\balpha!}}\binom{\balpha}{\bbeta}\EE[\bX^{\balpha-\bbeta}]\,.
\end{align*}
This means
\begin{equation}\label{eq:fhat-ghat}
\hat{g}_\bbeta = \sum_{\balpha \,:\, \bbeta \in \Gamma(\balpha)} \hat{f}_\balpha \sqrt{\frac{\bbeta!}{\balpha!}}\binom{\balpha}{\bbeta}\EE[\bX^{\balpha-\bbeta}]\,.
\end{equation}

Due to the symmetry of $f$, we have $\hat{g}_\bbeta = \hat{g}_{\bbeta'}$ whenever $\bbeta, \bbeta'$ are 
images of embeddings of the same $\Graph \in \cG_{\le D}$. Therefore, $g$ admits an expansion
 $g = \sum_{\Graph \in \cG_{\le D}} b_\Graph(g) \PolH_\Graph$ where
 $\PolH_\Graph$ is defined by Eq.~\eqref{eq:H_A}.

The coefficients $b_{\Graph}(g)$ and $\hat g_{\balpha}$ are related as follows. Let 
$|\Aut(\Graph)|$ denote the number of root-preserving automorphisms of 
$\Graph$, i.e., the number of embeddings $\phi \in \emb(\Graph)$ that 
induce a single element  $\balpha\in \NN^\Pairs$. For each $\Graph \in \cG_{\le D}$ 
there are $\frac{|\emb(\Graph)|}{|\Aut(\Graph)|}$ distinct elements 
$\balpha \in \NN^\Pairs$, and each has the same coefficient $\hat g_\balpha$.
Without loss of generality we can set $\hat g_\balpha=:|\Aut(\Graph)|/\sqrt{|\emb(\Graph)|}\,\tilde g_\Graph$ for some coefficients $\tilde g_\Graph$.
Therefore
\begin{align*}
g &= \sum_{\balpha\in\NN^{\Pairs}}\hat g_\balpha \PolH_{\balpha}  \\
& = \sum_{\Graph\in\cG_{\le D}}
\sum_{\balpha\in\NN^{\Pairs}}\hat g_\balpha \PolH_{\balpha}\cdot\frac{1}{|\Aut(\Graph)|}
\sum_{\phi\in\emb(\Graph)} \One_{\balpha=\sh(\phi;\Graph)}\\
& =  \sum_{\Graph\in\cG_{\le D}}\sum_{\phi\in\emb(\Graph)}
\sum_{\balpha\in\NN^{\Pairs}}\One_{\balpha=\sh(\phi;\Graph)}\frac{1}{\sqrt{|\emb(\Graph)|}} \tilde g_\Graph\PolH_{\balpha}\\
& =  \sum_{\Graph\in\cG_{\le D}}  \tilde g_\Graph\PolH_{\Graph}\, .
\end{align*}
Therefore we can conclude that $b_{\Graph}(g)=\tilde g_\Graph=\frac{\sqrt{|\emb(\Graph)|}}{|\Aut(\Graph)|}
 \hat g_\balpha$. This means
\begin{align*}
\|\hat g\|^2 = \sum_{\balpha \in \NN^\Pairs, |\balpha| \le D} \hat g_\balpha^2 &= 
\sum_{\Graph \in \cG_{\le D}} \frac{|\emb(\Graph)|}{|\Aut(\Graph)|} 
\left(\frac{|\Aut(\Graph)|}{\sqrt{|\emb(\Graph)|}}
 b_\Graph(g)\right)^2 \\
&= \sum_{\Graph \in \cG_{\le D}} |\Aut(\Graph)| \cdot b_{\Graph}(g)^2 \ge \|\bb(g) \|^2 \, .
\end{align*}
It now suffices to show $\|\bb(g)\|^2 = \Omega(1)$.

Our next step will be to relate the coefficients $\bb(g)$ to the coefficients $\bb(f)$. 
Using~\eqref{eq:fhat-ghat} along with the above relation between $\tilde g$ and $\hat g$, 
we have for any $B \in \cG_{\le D}$ and $\bbeta$ the image of some $\phi \in \emb(B)$,
\begin{equation}\label{eq:gtilde-fhat}
b_B(g) = \frac{\sqrt{|\emb(B)|}}{|\Aut(B)|} \hat g_\bbeta = \frac{\sqrt{|\emb(B)|}}{|\Aut(B)|} 
\sum_{\balpha \,:\, \bbeta \in \Gamma(\balpha)} \hat{f}_\balpha 
\sqrt{\frac{\bbeta!}{\balpha!}}\binom{\balpha}{\bbeta}\EE[\bX^{\balpha-\bbeta}]\,.
\end{equation}
Similarly to above, we have $\hat f_\balpha = \frac{|\Aut(A)|}{\sqrt{|\emb(A)|}} b_{A}(f)$ 
whenever $\balpha$ is the image of some $\phi \in \emb(A)$. In~\eqref{eq:gtilde-fhat}, 
the number of terms $\balpha \in \NN^{\Pairs}$ in the sum that correspond to a single 
$A \in \cG_{\le D}$ is $O(n^{|V(A)| - |V(B)|})$, recalling that 
$\bbeta \in \Gamma(\balpha)$ implies $\bbeta \le \balpha$. We also have the bounds
 $\sqrt{\frac{\bbeta!}{\balpha!}}\binom{\balpha}{\bbeta} = O(1)$, $\EE[\bX^{\balpha-\bbeta}] = 
 O(n^{-\frac{1}{2}(|E(A)| - |E(B)|)})$, $|\Aut(A)| = \Theta(1)$, and $|\emb(A)| = 
 \Theta(n^{|V(A)|-1})$. This means~\eqref{eq:gtilde-fhat} can be written as 
 $b_B(g) = \sum_A \zeta_{BA} b_A(f)$ for some coefficients $\zeta_{BA}$ (not depending on $f,g$)
  that satisfy
\begin{align*}
|\zeta_{BA}| &\le \sqrt{\frac{|\emb(B)|}{|\emb(A)|}} \cdot O(n^{|V(A)| - |V(B)|} \cdot n^{-\frac{1}{2}(|E(A)| - |E(B)|)}) \\
&=  O(n^{\frac{1}{2}(|V(A)| - |V(B)| - |E(A)| + |E(B)|)}) \\
&= O(1)\, ,
\end{align*}
where the last bound holds since $\bbeta \in \Gamma(\balpha)$ and
therefore $\balpha$ cannot have more components than $\bbeta$.

We have now shown $|\zeta_{BA}| = O(1)$. We can also see directly that $\zeta_{AA} = 1$ for all $A$, and $\zeta_{BA} = 0$ whenever $B \ne A$ and $|E(B)| \ge |E(A)|$.
This means
\[ b_B(g) = b_B(f) + \sum_{A \,:\, |E(A)| > |E(B)|} \zeta_{BA} b_A(f) \]
which we can rewrite in vector form as
\[ \bb(g) = \bzeta\,  \bb(f)\, , \] 
where $\bzeta= (\zeta_{A,B})_{A,B\in\cG_{\le D}}$. We order $\cG_{\le D}$ so that
the number of edges is non-decreasing along this ordering, and therefore $\bzeta$
is upper triangular with ones on the diagonal.
This implies $\det(\bzeta)=1$, and in particular $\bzeta$ is invertible. 
By Cramer's rule, letting $\bzeta_{(B,A)}$ denote the $(B,A)$-th minor,
\begin{align*}
(\bzeta^{-1})_{A,B} = (-1)^{i+j}\frac{\det(\bzeta_{(B,A)})}{\det(\bzeta)} = 
 (-1)^{i+j}\det(\bzeta_{(B,A)}) = O(1)\,,
 \end{align*}
where the last bound follows from the fact that the matrix dimensions are independent of $n$, and
$\max_{A,B\in\cG_{\le D}}|\zeta_{A,B}| = O(1)$. Therefore $\|\bzeta^{-1}\|_{\op} = O(1)$ and
\begin{align*}
\|\bb(g)\|_2 \ge  \|\bzeta^{-1}\|_{\op}^{-1}\,  \|\bb(f)\|_2\ge  \|\bzeta^{-1}\|_{\op}^{-1} = \Omega(1)\,.
\end{align*}
This concludes the proof.

\subsection{Proof of Lemma~\ref{lem:q-formula}: Explicit solution}

Throughout this proof, we will omit the subscript $n$ from $\bc$ and $\bM$.

For an arbitrary $q \in \RR[\bY]_{\le D}^\sym$, write the expansion $q = \sum_{A \in \cG_{\le D}} \hat{q}_{A} \PolH_{A}$. Recalling the definitions $c_{A} = \EE[\PolH_A(\bY) \cdot \psi(\theta_1)]$ and $M_{AB} = \EE[\PolH_A(\bY) \PolH_B(\bY)]$, we can rewrite the optimization problem as
\begin{align*}
\inf_{q \in \RR[\bY]^\sym_{\le D}} \EE[(q(\bY) - \psi(\theta_1))^2] &= \EE[\psi(\theta_1)^2] - \sup_{q \in \RR[\bY]^\sym_{\le D}} \EE[2 \, q(\bY) \cdot \psi(\theta_1) - q(\bY)^2] \\
&= \EE[\psi(\theta_1)^2] - \sup_{\hbq} \left(2 \langle \hbq, \bc \rangle - \langle \hbq, \bM \hbq \rangle \right).
\end{align*}
Since $\bM$ is positive definite by Lemma~\ref{lem:M-pd}, the map $\hbq \mapsto 2 \langle \hbq, \bc \rangle - \langle \hbq, \bM \hbq \rangle$ is strictly concave and is thus uniquely maximized at its unique stationary point $\hbq = \bM^{-1} \bc$.


\subsection{Proof of Lemma~\ref{lem:c-lim}: Limit of $\bc_n$}

Throughout this proof, we will omit the subscript $n$ from $\bc$ and its entries.

Our goal is to compute, for each $A\in\cG_{\le D}$,
\[ c_A = \EE[\PolH_A(\bY) \cdot \psi(\theta_1)] = 
\frac{1}{\sqrt{|\emb(A)|}} \sum_{\phi \in \emb(A)} \EE[\PolH_{\sh(\phi;A)}(\bY) \cdot \psi(\theta_1)]
 = \sqrt{|\emb(A)|} \EE[\PolH_\balpha(Y) \cdot \psi(\theta_1)] \, ,
 \]
where in the last step, $\balpha$ is the image of $A$ under  an arbitrary embedding $\phi$ 
and we have used the fact that $\EE[\PolH_\balpha(Y) \cdot \psi(\theta_1)]$ depends only
 on $A$ (not $\balpha$).
We have
\begin{equation}\label{eq:emb}
|\emb(A)| = \binom{n-1}{|V(A)|-1} (|V(A)|-1)! = n^{|V(A)|-1} (1 + O(n^{-1})) \, .
\end{equation}
For $\balpha$ the image of an embedding of $A$, we have by definition,
\[ \EE[\PolH_\balpha(\bY) \cdot \psi(\theta_1)] =
\EE\Big\{ \psi(\theta_1)\prod_{\bgamma \in C(\balpha)} (h_\bgamma(\bY) - \EE h_\gamma(\bY))
\Big\} \, . \]
If $A = \emptyset$ (the edgeless graph), we can see $c_A = \EE[\psi(\theta_1)]$ 
is a constant and we are done. If $A$ has a non-empty connected component 
that does not contain the root then we have $\EE[\PolH_\balpha(\bY) \cdot \psi(\theta_1)] = 0$ due to independence between components, and again we are done.

Finally, consider the case in which $A$ has a single non-empty component and this 
component contains the root. In this case, simplifying the above using the Hermite facts from Section~\ref{sec:hermite}, and recalling that $\bX:=\btheta\btheta^{\sT}/\sqrt{n}$,
\begin{align*}
&\EE[\PolH_\balpha(Y) \cdot \psi(\theta_1)] = \EE\big\{ \PolH(\theta_1) (h_\balpha(\bY) - \EE h_\balpha(\bY)) 
\big\}\\
&\qquad= \EE \left\{\psi(\theta_1) \left(\sum_{\bzero \le \balpha' \le \balpha}\sqrt{\frac{\balpha'!}{\balpha!}}\binom{\balpha}{\balpha'} X^{\balpha-\balpha'} h_{\balpha'}(\bZ) - \frac{1}{\sqrt{\balpha!}}\EE[\bX^\balpha]\right) \right\}\\
&\qquad= \EE \left\{\psi(\theta_1) \left(\sum_{\bzero \le \balpha' \le \balpha}\sqrt{\frac{\balpha'!}{\balpha!}}\binom{\balpha}{\balpha'} n^{-\frac{1}{2}|\balpha-\balpha'|} (\btheta\btheta^\sT)^{\balpha-\balpha'} h_{\balpha'}(\bZ) - \frac{1}{\sqrt{\balpha!}} n^{-\frac{1}{2}|\balpha|} \EE[(\btheta\btheta^\sT)^\balpha]\right)\right\}.
\end{align*}
Using orthogonality of the Hermite polynomials and recalling $h_0(z) = 1$, all the terms with $\balpha' \ne \bzero$ 
are zero in expectation, i.e.,
\begin{align*}
\EE[\PolH_\balpha(Y) \cdot \psi(\theta_1)] &= \EE\left\{\psi(\theta_1) \frac{n^{-\frac{1}{2}|\balpha|}}{\sqrt{\balpha!}} \left((\btheta\btheta^{\sT})^\balpha - \EE [(\btheta\btheta^{\sT})^\balpha]\right) \right\}\\
&= \const(A) \cdot n^{-\frac{1}{2}|\balpha|} = \const(A) \cdot n^{-\frac{1}{2}|E(A)|} \, .
\end{align*}

Putting it all together, we get
\[ c_A = \const(A) \cdot n^{\frac{1}{2}(|V(A)| - 1 - |E(A)|)} \cdot\big(1+O(n^{-1})\big)\, . \]
Recall from above that $c_A = 0$ unless $A$ has a single connected component and this
 component contains the root; we therefore restrict to $A$ of this type in what follows. 
 Due to connectivity, any such $A$ has $|V(A)| \le |E(A)| + 1$, implying $c_A = \const(A) + O(n^{-1})$
  as desired. Now suppose furthermore that $A \in \cG_{\le D} \setminus \cT_{\le D}$, i.e.,
   $A$ contains a cycle (where we consider a self-loop or double-edge to be a cycle). 
   In this case, $|V(A)| < |E(A)| + 1$, implying $c_A = O(n^{-1/2})$ as desired.

\subsection{Proof of Lemma~\ref{lem:M-lim}: Limit of $\bM_n$}
\label{sec:pf-M-lim}

Throughout this proof, we will omit the subscript $n$ from $\bM$.

We will need to reason about the different possible intersection patterns between two rooted graphs 
$A, B \in \cG_{\le D}$. To this end, define a \emph{pattern} 
to be a rooted colored multigraph where no vertices are isolated 
except possibly the root, every edge is colored either red or blue, 
the edge-induced subgraph of red edges (with the root always included) is isomorphic 
(as a rooted graph) to $A$, and the edge-induced subgraph of blue edges is isomorphic to $B$. 
Let $\pat(A,B)$ denote the set of such patterns, up to isomorphism of rooted colored graphs.
The number of such patterns is a constant depending on $A,B$. For $\Pi \in \pat(A,B)$, 
let $\emb(\Pi)$ denote the set of embeddings of $\Pi$, where an embedding is defined 
to be a pair $(\phi_A,\phi_B)$ with $\phi_A \in \emb(A), \phi_B \in \emb(B)$ such that the induced 
colored graph on vertex set $[n]$ is isomorphic to $\Pi$. We let 
$|V(\Pi)|$ denote the number of vertices in the pattern (including the root).

We can write $M_{AB} = \EE[\PolH_A(\bY) \PolH_B(\bY)]$ as 
\begin{align*}
M_{AB}  &= \frac{1}{\sqrt{|\emb(A)| \cdot |\emb(B)|}} 
\sum_{\phi_A \in \emb(A), \phi_B \in \emb(B)} \EE[\PolH_{\sh(\phi_A;A)}(\bY) 
\PolH_{\sh(\phi_B;B)}(\bY)] \\
&= \frac{1}{\sqrt{|\emb(A)| \cdot |\emb(B)|}} \sum_{\Pi \in \pat(A,B)}
 \sum_{(\phi_A,\phi_B) \in \emb(\Pi)} \EE[\PolH_{\sh(\phi_A;A)}(\bY) \PolH_{\sh(\phi_B;B)}(\bY)] \\
&= \frac{1}{\sqrt{|\emb(A)| \cdot |\emb(B)|}} \sum_{\Pi \in \pat(A,B)} |\emb(\Pi)| 
\cdot \EE[\PolH_\balpha(Y) \PolH_\bbeta(\bY)] \, ,
\end{align*}
where in the last line, $(\balpha,\bbeta) = (\sh(\phi_A;A),\sh(\phi_B;B))$ is the shadow of an arbitrary 
embedding $(\phi_A,\phi_B) \in \emb(\Pi)$, and we have used the fact that
 $\EE[\PolH_\balpha(\bY) \PolH_\bbeta(\bY)]$ depends only on $\Pi$ (not $\balpha,\bbeta$). 
 Recall from~\eqref{eq:emb} that
\[ |\emb(A)| = n^{|V(A)|-1} (1 + O(n^{-1})) \, , \]
and we similarly have
\[ |\emb(\Pi)| = \const(\Pi) \cdot \binom{n-1}{|V(\Pi)|-1} = \const(\Pi) \cdot n^{|V(\Pi)|-1} (1 + O(n^{-1})) \, . \]
Now compute using the Hermite facts from Section~\ref{sec:hermite},
 and recalling that $\bX:=\btheta\btheta^{\sT}/\sqrt{n}$,
\begin{align*}
\EE[\PolH_\balpha(\bY)&\PolH_\bbeta(\bY)] = 
\EE \left(\prod_{\bgamma \in C(\balpha)} \left(h_\bgamma(\bY) - \EE h_\bgamma(\bY)\right)\right)
 \left(\prod_{\bdelta \in C(\bbeta)} \left(h_\bdelta(\bY) - \EE h_\bdelta(\bY)\right)\right) \\
&= \EE \left(\prod_{\bgamma \in C(\balpha)} \left(\sum_{\bzero \le \bgamma' \le \bgamma}\sqrt{\frac{\bgamma'!}{\bgamma!}}\binom{\bgamma}{\bgamma'} \bX^{\bgamma-\bgamma'} h_{\bgamma'}(\bZ) - \frac{1}{\sqrt{\bgamma!}}\EE[\bX^\bgamma]\right)\right) \cdot \\
&\qquad\qquad \left(\prod_{\bdelta \in C(\bbeta)}\left(\sum_{\bzero \le \bdelta' \le \bdelta}\sqrt{\frac{\bdelta'!}{\bdelta!}}\binom{\bdelta}{\bdelta'} \bX^{\bdelta-\bdelta'} h_{\bdelta'}(\bZ) - \frac{1}{\sqrt{\bdelta!}}\EE[\bX^\bdelta]\right)\right) \\
&= \EE \left(\prod_{\bgamma \in C(\balpha)} \left(\sum_{\bzero \le \bgamma' \le \bgamma}\sqrt{\frac{\bgamma'!}{\bgamma!}}\binom{\bgamma}{\bgamma'} n^{-\frac{1}{2}|\bgamma-\bgamma'|} (\btheta\btheta^\sT)^{\bgamma-\bgamma'} h_{\bgamma'}(\bZ) - \frac{1}{\sqrt{\bgamma!}} n^{-\frac{1}{2}|\bgamma|} \EE[(\btheta\btheta^\sT)^\bgamma]\right)\right) \cdot \\
&\qquad\qquad \left(\prod_{\bdelta \in C(\bbeta)}\left(\sum_{\bzero \le \bdelta' \le \bdelta}\sqrt{\frac{\bdelta'!}{\bdelta!}}\binom{\bdelta}{\bdelta'} n^{-\frac{1}{2}|\bdelta-\bdelta'|} (\btheta\btheta^\sT)^{\bdelta-\bdelta'} h_{\bdelta'}(\bZ) - \frac{1}{\sqrt{\bdelta!}} n^{-\frac{1}{2}|\bdelta|} \EE[(\btheta\btheta^\sT)^\bdelta]\right)\right).
\end{align*}
If $\balpha$ has a non-empty (i.e., containing at least one edge) connected component that is vertex-disjoint from all non-empty 
connected components of $\bbeta$ or vice-versa, then we have $\EE[\PolH_\balpha(\bY) \PolH_\bbeta(\bY)] = 0$ from the first line above (due to independence between components). 
Otherwise we say that $\Pi$ has ``no isolated components'' and denote by $\pat^*(A,B)$ the set of 
patterns with this property. For $\Pi \in \pat^*(A,B)$, using orthogonality of the Hermite 
polynomials, the expansion of the expression above has a unique leading (in $n$) 
term where $\bgamma', \bdelta'$ are as large as possible: namely, 
$\bgamma' = \bgamma \wedge \bbeta$ and $\bdelta' = \bdelta \wedge \balpha$ where $\wedge$ 
denotes entrywise minimum between elements of $\NN^\Pairs$. 
Therefore there exists $\const(\Pi) \in \RR$ independent of $n$ such that
\begin{align*}
\EE[\PolH_\balpha(\bY) \PolH_\bbeta(\bY)] &= \big(\const(\Pi) +O(n^{-1/2})\big)
\cdot n^{-\frac{1}{2}|\balpha - (\balpha \wedge \bbeta)| - \frac{1}{2}|\bbeta - (\balpha \wedge \bbeta)|} \\
&= \big(\const(\Pi) +O(n^{-1/2})\big)\cdot n^{-\frac{1}{2}|\balpha \symd \bbeta|}\, ,
\end{align*}
where the symmetric difference is defined as $(\balpha \symd \bbeta)_{ij} := |\alpha_{ij} - \beta_{ij}|$.

Putting it all together, and recalling $\EE[\PolH_\balpha(\bY) \PolH_\bbeta(\bY)] = 0$ whenever 
$\Pi \in \pat(A,B) \setminus \pat^*(A,B)$ (and eventually redefining  $\const(\Pi)$)
\begin{align*}
M_{AB} &=  n^{-\frac{1}{2}(|V(A)|-1)-\frac{1}{2}(|V(B)|-1)} \sum_{\Pi \in \pat^*(A,B)}
 \big(\const(\Pi)+O(n^{-1/2})\big) \cdot n^{(|V(\Pi)|-1) - \frac{1}{2}|\alpha \symd \beta|}  \\
&=\big(\const(A,B)+O(n^{-1/2})\big) \cdot n^{\varphi(A,B)}
\end{align*}
where
\begin{align*}
\varphi(A,B) :=& -\frac{1}{2}(|V(A)| + |V(B)|) + \max_{\Pi \in \pat^*(A,B)}
\Big( |V(\Pi)| - \frac{1}{2}|\balpha \symd \bbeta| \Big)\\
=&\; \frac{1}{2} \max_{\Pi \in \pat^*(A,B)} \Big(
|V(\balpha) \setminus V(\bbeta)| + |V(\bbeta) \setminus V(\balpha)| - |\balpha \setminus \bbeta| - 
|\bbeta \setminus \balpha|\Big)\, ,
\end{align*}
where $(\balpha \setminus \bbeta)_{ij} := \max\{0,\alpha_{ij} - \beta_{ij}\}$, and $V(\balpha)$ always
 includes the root by convention (see Section~\ref{sec:notation}). To complete the proof, 
 it remains to show $\varphi(A,B) \le 0$ for all $A,B$, and furthermore $\varphi(A,B) \le -1/2$ 
 in the case $A \in \cT_{\le D}, B \in \cG_{\le D} \setminus \cT_{\le D}$.

For any $A,B \in \cG_{\le D}$, $\Pi \in \pat^*(A,B)$, and $(\balpha,\bbeta)=
(\sh(\phi;A),\sh(\phi;B))$ for $\phi$ an embedding of $\Pi$, each non-empty connected component
 $\bgamma$ of $\balpha \setminus \bbeta$ spans at most $|\bgamma| + 1$ vertices.
  Of these vertices, none belong to $V(\bbeta) \setminus V(\balpha)$ and due to the 
  ``no isolated components'' constraint imposed by $\pat^*(A,B)$, at most $|\bgamma|$ of 
  them (i.e., all but one) can belong to $V(\balpha) \setminus V(\bbeta)$. Every vertex in 
  $V(\balpha) \setminus V(\bbeta)$ belongs to some non-empty component of 
  $\balpha \setminus \bbeta$ (since the root belongs to $V(\balpha) \cap V(\bbeta)$), 
  so we have $|V(\balpha) \setminus V(\bbeta)| \le |\balpha \setminus \bbeta|$. Similarly, 
  $|V(\bbeta) \setminus V(\balpha)| \le |\bbeta \setminus \balpha|$, implying $\varphi(A,B) \le 0$ 
  as desired.

In order to have equality $\varphi(A,B) = 0$ in the above argument, every non-empty connected 
component of $\balpha \setminus \bbeta$ must be a simple (i.e., without self-loops or multi-edges) 
tree that spans only one vertex in $V(\bbeta)$, and similarly $\bbeta \setminus \balpha$ must 
be a simple tree that spans only one vertex in $V(\balpha)$. It remains to show that this is 
impossible when $A \in \cT_{\le D}, B \in \cG_{\le D} \setminus \cT_{\le D}$. This
 means $A$ is a simple rooted tree. We can assume $A$ has at least one edge, or else the ``no isolated components'' property must fail (since $A$ has no non-empty components).
We will consider a few different cases for $B$.

First consider the case where the root is isolated in $B$. This means $\balpha \setminus \bbeta$ 
has a non-empty connected component containing the root $\circ$. Since $A$ is a simple rooted 
tree and due to ``no isolated components,'' this component also contains a vertex $u$ belonging 
to some non-empty component of $\bbeta$. However, now we have a non-empty component of 
$\balpha \setminus \bbeta$ that spans at least two vertices in $V(\bbeta)$, namely 
$\circ$ and $u$, and so by the discussion above, it is impossible to have equality $\varphi(A,B) = 0$.

Next consider the case where the root is not isolated in $B$, but $B$ has multiple
 connected components. Let $\bgamma$ be a non-empty component of $\bbeta$ that does
  not contain the root. Due to ``no isolated components,'' $\bgamma$ must span some (non-root)
   vertex $u \in V(\balpha)$. Since $\balpha$ is a simple rooted tree, there is a unique path in 
   $\balpha$ from $u$ to $\circ$. Recall that $\circ \in V(\bbeta)$ by convention, so both endpoints 
   of this path belong to $V(\bbeta)$. Also, this path must contain an edge in $\balpha \setminus \bbeta$,
    or else $\bgamma$ would be connected to the root. This means that along the path, we can 
    find a non-empty component of $\balpha \setminus \bbeta$ that spans two different vertices
     in $V(\bbeta)$, and so it is impossible to have equality $\varphi(A,B) = 0$.

The last case to consider is where $B$ contains a cycle. This cycle must contain 
at least one edge in $\bbeta \setminus \balpha$, since $\balpha$ has no cycles. 
Recall from above that in order to have $\varphi(A,B) = 0$, every non-empty component 
of $\bbeta \setminus \balpha$ must be a simple tree, which means the cycle also contains
 at least one edge in $\balpha$. Now we know the cycle in $\bbeta$ has at least one edge 
 in $\bbeta \setminus \balpha$ and at least one edge in $\balpha$, but this means 
 $\bbeta \setminus \balpha$ has a non-empty connected component that spans at least 
 two vertices in $V(\balpha)$, and so it is impossible to have $\varphi(A,B) = 0$.

\subsection{Limit of $\<\bc_n,\bM_n^{-1} \bc_n\>$}

\begin{lemma}\label{lem:limit-exists}
Under the assumptions of Proposition \ref{prop:tree-vs-poly},
and with $\bM_{\infty}$, $\bc_{\infty}$ as in the statement of Lemmas~\ref{lem:c-lim}
and~\ref{lem:M-lim},
we have
\begin{align}
\lim_{n\to\infty}\<\bc_n,\bM_n^{-1}\bc_n\> = \<\bc_\infty,\bM_\infty^{-1}\bc_\infty\> \, .
\end{align}
\end{lemma}
\begin{proof}
From Lemma~\ref{lem:M-pd} we have $\lambda_{\mathrm{min}}(\bM_n) = \Omega(1)$. As a 
direct consequence of Lemma~\ref{lem:M-lim},
\[ \|\bM_n-\bM_{\infty}\|_F = O(n^{-1/2}) \, . \]
This means, using $\|\;\cdot\;\|_{\op}$ for matrix operator norm,
\[ \lambda_{\mathrm{min}}(\bM_{\infty}) \ge \lambda_{\mathrm{min}}(\bM_n) -
 \|\bM_n - \bM_{\infty}\|_{\op} \ge \lambda_{\mathrm{min}}(\bM_n) - \|\bM_n- \bM_{\infty}\|_F 
 \ge\const>0 \, , \]
implying that $\bM_{\infty}$ is symmetric positive definite and thus invertible. 
The result is now immediate from Lemmas~\ref{lem:c-lim} and~\ref{lem:M-lim} because 
$\bM_n$ has fixed dimension and the entries of $\bM^{-1}_n = \mathrm{adj}(\bM_n)/\mathrm{det}(\bM_n)$ 
are differentiable functions of the entries of $\bM_n$ (in a neighborhood of $\bM_{\infty}$).
\end{proof}

\subsection{Proof of Lemma~\ref{lem:r-success}: Accuracy of $r$}

First we claim that $\<\bd_{\infty},\bP_{\infty}^{-1} \bd_{\infty}\> =
 \<\bc_{\infty}, \bM_{\infty}^{-1} \bc_{\infty}\>$. Recall the block form for 
$\bM_n$ and $\bc_n$ in Eq.~\eqref{eq:block-form}, and that
 $\bR_{\infty} = 0$ per Lemma \ref{lem:M-lim} and therefore
\[ \bM_{\infty} = \left[\begin{array}{cc} \bP_{\infty} & \bzero \\ 
\bzero & \bQ_{\infty} \end{array}\right] \qquad\text{and}\qquad 
\bM_{\infty}^{-1} = \left[\begin{array}{cc} \bP_{\infty}^{-1} & 0 \\ 0 & \bQ_{\infty}^{-1} \end{array}\right]. \]
The claim follows because $\be_{\infty} = 0$ (see Eq.~\eqref{eq:block-form}
 and Lemma \ref{lem:c-lim}).

Now as $n \to \infty$ we have
\[ \EE[r(\bY) \cdot \psi(\theta_1)] = \< \hbr, \bd_n \> = 
\<\bd_n , \bP_{\infty}^{-1} \bd_{\infty}\> \,\to\, 
\<\bd_{\infty}, \bP_{\infty}^{-1} \bd_{\infty}\>\, , \]
and
\[ \EE[r(\bY)^2] = \<\hbr, \bP_n \hbr\> = \<\bd_{\infty},
 \bP_{\infty}^{-1} \bP_n \bP^{-1}_{\infty} \bd_{\infty}\> \,\to\, \<\bd_{\infty}, \bP_{\infty}^{-1}
  \bd_{\infty}\>\, , \]
completing the proof.

\subsection{Proof of Lemma~\ref{lem:change-basis}: Change-of-basis}

Fix $A \in \cT_{\le D}$. If $A = \emptyset$ then $\PolH_A = \Poly_A = 1$ and the proof is 
complete, so suppose $A \ne \emptyset$. Recall that our goal is to take
\[ \PolH_A(\bY) := \frac{1}{\sqrt{|\emb(A)|}} \sum_{\phi \in \emb(A)} \PolH_{\sh(\phi;A)}(\bY) \]
and approximate it in the basis $\{\Poly_B \,:\, B \in \cT_{\le D}\}$, where
\[ \Poly_B(\bY) := \frac{1}{\sqrt{|\nr(B)|}} \sum_{\phi \in \nr(B)} \bY^\phi \, , \]
where we use the shorthand
\[ \bY^\phi := \bY^{\sh(\phi;B)} = \prod_{(i,j) \in E(B)} Y_{\phi(i),\phi(j)} \, . \]
Since $A$ is a tree with at least one edge,
\begin{align}
\PolH_A(\bY) &= \frac{1}{\sqrt{|\emb(A)|}}
 \sum_{\phi \in \emb(A)}(h_{\sh(\phi;A)}(\bY) - \E h_{\sh(\phi)}(\bY)) \nonumber \\
 &= 
 \frac{1}{\sqrt{|\emb(A)|}} \sum_{\phi \in \emb(A)} \bY^\phi - \sqrt{|\emb(A)|} \EE \bY^{\phi^*}\, ,
\label{eq:H_A-exp}
\end{align}
for an arbitrary $\phi^* \in \emb(A)$.

We first focus on the non-random term $\sqrt{|\emb(A)|} \EE \bY^{\phi^*}$, which will be easy 
to expand in the basis $\{\Poly_B\}$ because $\Poly_\emptyset = 1$. Let
\[ m_A = \lim_{n \to \infty} \sqrt{|\emb(A)|} \EE \bY^{\phi^*} \, , \]
noting that this limit exists by combining~\eqref{eq:emb} with the fact
\[ \EE \bY^{\phi^*} = \const(A) \cdot n^{-\frac{1}{2}|E(A)|} = 
\const(A) \cdot n^{-\frac{1}{2}(|V(A)|-1)} \, . \]
We can now write
\[ \sqrt{|\emb(A)|} \EE \bY^{\phi^*} = m_A \Poly_\emptyset + \PolErr_{A,0} \]
where $\PolErr_{A,0} := \sqrt{|\emb(A)|} \EE \bY^{\phi^*} - m_A$ satisfies 
$\PolErr_{A,0} = o(1)$ and so (being non-random) $\EE[\PolErr_{A,0}^2] = o(1)$.

The error term $\PolErr_A$ will be the sum of $K$ terms 
$\PolErr_A = \sum_{i=1}^K \PolErr_{A,i}$, with $K$ independent of $n$. It suffices to show 
$\EE[\PolErr_{A,i}^2]=o(1)$ for each
 individually, as this implies by the triangle inequality
\begin{equation}\label{eq:triangle-error}
\EE[\PolErr_A^2]^{1/2} \le \sum_{i=1}^K\EE[\PolErr_{A,i}^2]^{1/2}=o(1) \, .
\end{equation}

We next handle the more substantial term in~\eqref{eq:H_A-exp}, namely
\begin{align}\label{eq:G}
G_A(\bY)& := \frac{1}{\sqrt{|\emb(A)|}} \sum_{\phi \in \emb(A)} \bY^\phi \\
&= \frac{1}{\sqrt{|\emb(A)|}} \sum_{\phi \in \nr(A)} \bY^\phi - \frac{1}{\sqrt{|\emb(A)|}}
 \sum_{\phi \in \nr(A) \setminus \emb(A)} \bY^\phi \, .\nonumber
\end{align}
For the first term,
\begin{align*}
 \frac{1}{\sqrt{|\emb(A)|}} \sum_{\phi \in \nr(A)} \bY^\phi& = 
\sqrt{\frac{|\nr(A)|}{|\emb(A)|}} \Poly_A(\bY)\\
& =
 \Poly_A(\bY) + \left(\sqrt{\frac{|\nr(A)|}{|\emb(A)|}} - 1\right) \Poly_A(\bY)
  =: \Poly_A(\bY) + \PolErr_{A,1}(\bY) \, . 
  \end{align*}
Since $\lim_{n \to \infty} (|\nr(A)|/|\emb(A)|) = 1$ (in fact, embeddings make up a 
$1-o(1)$ fraction of \emph{all} labelings, not just non-reversing ones) and 
$\EE[\Poly_A^2] = O(1)$ (similarly to the proof of Lemma~\ref{lem:E_A2} below), we have 
$\EE[\PolErr_{A,1}^2] = o(1)$.

Returning to the remaining term in~\eqref{eq:G}, we partition 
$\nr(A) \setminus \emb(A)$ into  two sets
 $L_1 = L_1(A)$ and $L_2 = L_2(A)$, defined as follows.
  If the multigraph 
 $\balpha := \sh(\phi)$ has no triple-edges or higher 
 (i.e., $\alpha_{ij} \le 2$ for all $i \le j$) and has no cycles 
 (namely, the simple graph $\bbeta$ defined by $\beta_{ij} = \One_{\alpha_{ij} \ge 1}$ has no cycles) 
 then we let $\phi \in L_1$; otherwise $\phi \in L_2$. The remaining term to handle is
\begin{align}
\frac{1}{\sqrt{|\emb(A)|}} \sum_{\phi \in \nr(A) \setminus \emb(A)} \bY^\phi &= 
\frac{1}{\sqrt{|\emb(A)|}} \sum_{\phi \in L_1} \bY^\phi +
 \frac{1}{\sqrt{|\emb(A)|}} \sum_{\phi \in L_2} \bY^\phi \nonumber \\
&=: \frac{1}{\sqrt{|\emb(A)|}} \sum_{\phi \in L_1} \bY^\phi + \PolErr_{A,2}(\bY) \, .
\label{eq:nr-emb-term}
\end{align}

\begin{lemma}\label{lem:E_A2}
$\EE[\PolErr_{A,2}^2] = o(1)$.
\end{lemma}
\begin{proof}
 We define the set of patterns $\Pi \in \pat_2(A,B)$
 between two rooted trees $A,B \in \cT_{\le D}$ as Section~\ref{sec:pf-M-lim}, but now
restricted to labelings in $L_2$.
 Namely edge-induced subgraph of red edges in $\Pi\in \pat_2(A,B)$ should not be isomorphic to $A$ but
  rather isomorphic to the image of some $\phi \in L_2(A)$ (and similarly for the blue 
  edges and $B$). We  write $(\phi_1,\phi_2) \in \emb(\Pi)$ to denote an
   embedding of this pattern into vertex set $[n]$. Similarly to Section~\ref{sec:pf-M-lim}, 
   compute
\begin{align}
\EE[\PolErr_{A,2}^2] &= \frac{1}{|\emb(A)|} \sum_{\phi_1,\phi_2 \in L_2(A)} 
\EE[\bY^{\phi_1} \bY^{\phi_2}] \nonumber \\
&= \frac{1}{|\emb(A)|} \sum_{\Pi \in \pat_2(A,A)} \, 
\sum_{(\phi_1,\phi_2) \in \emb(\Pi)} \EE[\bY^{\phi_1} \bY^{\phi_2}] \nonumber \\
&= \frac{1}{|\emb(A)|} \sum_{\Pi \in \pat_2(A,A)} |\emb(\Pi)| \cdot \EE[\bY^{\balpha+\bbeta}] \nonumber
\intertext{where $(\balpha,\bbeta) = (\sh(\phi_1;A),\sh(\phi_2;A))$ is the image
 of an arbitrary embedding of $\Pi$}
&= \Theta(n^{1-|V(A)|}) \sum_{\Pi \in \pat_2(A,A)} O(n^{|V(\Pi)|-1} \cdot n^{-\frac{1}{2}\odd(\Pi)}) \nonumber
\intertext{where $\odd(\Pi)$ denotes the number of edges $i \le j$ for which $\alpha_{ij} + \beta_{ij}$ is odd, 
which depends only on $\Pi$ (not $\balpha,\bbeta$). Now since $A$ is a tree,
 $|E(\Pi)| = 2 |E(A)| = 2(|V(A)| - 1)$ and so the above becomes}
&= \sum_{\Pi \in \pat_2(A,A)} O(n^{|V(\Pi)|-\frac{1}{2}|E(\Pi)|-1-\frac{1}{2}\odd(\Pi)}) \label{eq:L2-exp}
\end{align}
and so it remains to show that the exponent is strictly negative for any $\Pi \in \pat_2(A,A)$.

Fix $\Pi \in \pat_2(A,A)$. Let $s$ denote the number of single-edges in $\Pi$ (ignoring the edge colors), let $d$ denote the number of double-edges, let $o$ denote the number of $k$-edges for $k \ge 3$ odd, and let $e$ denote the number of $k$-edges for $k \ge 4$ even. By definition we have
\[ \odd(\Pi) = s + o \, , \]
\[ |E(\Pi)| \ge s + 2d + 3o + 4e \, . \]
Also, since $A$ (and therefore $\Pi$) is connected,
\[ |V(\Pi)| \le s+d+o+e+1-\One_{\mathrm{cycle}} \]
where $\One_{\mathrm{cycle}}$ is the indicator that $\Pi$ (viewed as a simple graph by replacing each multi-edge by a single-edge) contains a cycle. Now the exponent in~\eqref{eq:L2-exp} is
\begin{align*}
|V(\Pi)|-\frac{1}{2}|E(\Pi)|-1-\frac{1}{2}\odd(\Pi)
&\le (s+d+o+e+1-\One_{\mathrm{cycle}}) - \frac{1}{2}(s+2d+3o+4e) - 1 - \frac{1}{2}(s+o) \\
&= -(o+e+\One_{\mathrm{cycle}}) \, .
\end{align*}
By the definition of $L_2$, we must either have $o + e \ge 1$ or $\One_{\mathrm{cycle}} = 1$, which means the exponent is $\le -1$, completing the proof.
\end{proof}

It remains to handle the first term in Eq.~\eqref{eq:nr-emb-term}. For $\phi \in L_1(A)$, 
define the ``skeleton'' $\sk(\phi) \in \NN^\Pairs$ to be the subgraph of $\sh(\phi)$ obtained from 
$\balpha := \sh(\phi;A)$ as follows: delete all multi-edges (i.e., whenever $\alpha_{ij} \ge 2$,
 set $\alpha_{ij} = 0$) and then take the connected component of the root (vertex 1).
  Using the definition of $L_1(A)$ and recalling that $\phi \in L_1(A)$ is not an embedding, 
  note that $\sk(\phi)$ is always a simple tree with strictly fewer edges than $A$. 
  The remaining term to handle is
\begin{align}
\frac{1}{\sqrt{|\emb(A)|}} \sum_{\phi \in L_1} \bY^\phi &= \frac{1}{\sqrt{|\emb(A)|}}
 \sum_{\phi \in L_1} C_\phi \bY^{\sk(\phi)} + \frac{1}{\sqrt{|\emb(A)|}} \sum_{\phi \in L_1} \left(\bY^{\phi} - C_\phi \bY^{\sk(\phi)} \right) \nonumber \\
&=: \frac{1}{\sqrt{|\emb(A)|}} \sum_{\phi \in L_1} C_\phi \bY^{\sk(\phi)} + \PolErr_{A,3}(Y)
\label{eq:L1-term}
\end{align}
where
\[ C_\phi := \EE[\bY^{\sh(\phi)-\sk(\phi)}] \, . \]

\begin{lemma}
$\EE[\PolErr_{A,3}^2] = o(1)$.
\end{lemma}

\begin{proof}
Define $\pat_1(A,B)$ as in the proof of Lemma~\ref{lem:E_A2} except with $L_1$ in 
place of $L_2$. Our goal is to bound
\begin{align*}
\EE[\PolErr_{A,3}^2] &= \frac{1}{|\emb(A)|} \sum_{\phi_1,\phi_2 \in L_1(A)} 
\EE\left(\bY^{\phi_1} - C_{\phi_1} \bY^{\sk(\phi_1)} \right)\left(\bY^{\phi_2} - C_{\phi_2} \bY^{\sk(\phi_2)} \right) \\
&= \frac{1}{|\emb(A)|} \sum_{\Pi \in \pat_1(A,A)} \; \sum_{(\phi_1,\phi_2) \in \emb(\Pi)}
 \EE\left(\bY^{\phi_1} - C_{\phi_1} \bY^{\sk(\phi_1)} \right)\left(\bY^{\phi_2} - 
 C_{\phi_2} \bY^{\sk(\phi_2)} \right).
\end{align*}
For $\phi \in L_1(A)$, letting $\osk(\phi) = \sh(\phi) - \sk(\phi)$ 
(throughout this proof we adopt the shorthand $\sh(\phi)=\sh(\phi;A)$ since the base graph 
$A$ is fixed),
\begin{align*}
\bY^{\phi} - C_\phi \bY^{\sk(\phi)} &= \bY^{\sk(\phi)} \left(\bY^{\osk(\phi)} - 
\EE \bY^{\sh(\phi)-\sk(\phi)}\right) \\
&= \left(\sum_{0 \le \balpha \le \sk(\phi)} \bX^{\sk(\phi)-\balpha} \bZ^\balpha \right) 
\sum_{0 \le \bbeta \le \osk(\phi)} \left(\bX^{\osk(\phi)-\bbeta} \bZ^\bbeta - 
\EE [\bX^{\osk(\phi)-\bbeta} \bZ^\bbeta] \right).
\end{align*}
This means
\begin{align}
\label{eq:L1-expansion}
&\EE\left(\bY^{\phi_1} - C_{\phi_1} \bY^{\sk(\phi_1)} \right)\left(\bY^{\phi_2} - C_{\phi_2} \bY^{\sk(\phi_2)} \right) \\
&= \sum_{\substack{\bzero \le \balpha \le \sk(\phi_1) \\ \bzero \le \bbeta \le \osk(\phi_1) \\ \bzero \le \bgamma \le \sk(\phi_2) \\
 \bzero \le \bdelta \le \osk(\phi_2)}} \bX^{\sk(\phi_1)-\balpha} \bZ^\balpha \bX^{\sk(\phi_2)-\bgamma} 
 \bZ^\bgamma \left(\bX^{\osk(\phi_1)-\bbeta} \bZ^\bbeta - \EE [\bX^{\osk(\phi_1)-\bbeta} \bZ^\bbeta] \right) 
 \left(\bX^{\osk(\phi_2)-\bdelta} \bZ^\bdelta - \EE [\bX^{\osk(\phi_2)-\bdelta} \bZ^\bdelta] \right). \nonumber
\end{align}
We will next claim that certain terms in the sum in~\eqref{eq:L1-expansion} are zero.
 First note that we must have $\alpha_{ij} + \beta_{ij} + \gamma_{ij} + \delta_{ij}$
  is even for every edge $i \le j$, or else the corresponding term in~\eqref{eq:L1-expansion} 
  is zero due to the $Z$ factors. Therefore, for $(\phi_1,\phi_2) \in \emb(\Pi)$, 
  the value of~\eqref{eq:L1-expansion} is $O(n^{-\frac{1}{2} \odd(\Pi)})$ where, 
  recall, $\odd(\Pi)$ denotes the number of odd-edges in $\Pi$ (i.e., the number of edges 
  $i \le j$ for which $\sh(\phi_1)_{ij} + \sh(\phi_2)_{ij}$ is odd).

We will improve the bound $O(n^{-\frac{1}{2} \odd(\Pi)})$ in certain cases. 
Let $P(\Pi)$ denote the property that $\Pi$ has no triple-edges or higher 
(i.e., $\sh(\phi_1)_{ij} + \sh(\phi_2)_{ij} \le 2$ for all $i \le j$) and no cycles
 (when $\Pi$ is viewed as a simple graph by replacing multi-edges by single-edges). 
 We claim that~\eqref{eq:L1-expansion} is $O(n^{-\frac{1}{2} \odd(\Pi) - 1})$ whenever 
 $P(\Pi)$ holds. To prove this, first identify the ``bridges,'' i.e., edges $(i,j)$
  for which $\sh(\phi_1)_{ij} = 2$ and one endpoint of $(i,j)$ belongs to $V(\sk(\phi_1))$. 
  At least one bridge must exist by the definition of $L_1$. Using the definition of 
  $P(\Pi)$, $\osk(\phi_1)$ shares no vertices with $\sh(\phi_2)$ except possibly one endpoint 
  of each bridge. As a result, some bridge $(i,j)$ must have $\beta_{ij} = 0$ or else the
   corresponding term in~\eqref{eq:L1-expansion} is zero; to see this, note that if every 
   bridge has $\beta_{ij} = 2$ then the factor $(\bX^{\osk(\phi_1)-\bbeta} \bZ^\bbeta - 
   \EE [\bX^{\osk(\phi_1)-\bbeta} \bZ^\bbeta])$ has mean zero and is independent from the other factors 
   in~\eqref{eq:L1-expansion}. This gives the desired improvement to $O(n^{-\frac{1}{2} \odd(\Pi) - 1})$.

We now have
\begin{align*}
\EE[\PolErr_{A,3}^2] &= \frac{1}{|\emb(A)|} \sum_{\Pi \in \pat_1(A,A)} \; \sum_{(\phi_1,\phi_2) \in \emb(\Pi)} O(n^{-\frac{1}{2} \odd(\Pi) - \One_{P(\Pi)}}) \\
&= \sum_{\Pi \in \pat_1(A,A)} O(n^{-(|V(A)|-1)+(|V(\Pi)|-1)-\frac{1}{2} \odd(\Pi) - \One_{P(\Pi)}}) \\
&= \sum_{\Pi \in \pat_1(A,A)} O(n^{-|V(A)|+|V(\Pi)|-\frac{1}{2} \odd(\Pi) - \One_{P(\Pi)}})
\intertext{and now repeating the argument from the proof of Lemma~\ref{lem:E_A2}}
&= \sum_{\Pi \in \pat_1(A,A)} O(n^{|V(\Pi)|-\frac{1}{2}|E(\Pi)|-1-\frac{1}{2} \odd(\Pi) - \One_{P(\Pi)}}) \\
&\le \sum_{\Pi \in \pat_1(A,A)} O(n^{-(o+e+\One_{\mathrm{cycle}} + \One_{P(\Pi)})}) \\
&\le O(n^{-1}) \, ,
\end{align*}
completing the proof.
\end{proof}

It remains to handle the first term in~\eqref{eq:L1-term}. For each $\phi \in L_1$, $\sh(\phi)$ is isomorphic (as a rooted multigraph) to some connected multigraph that is a rooted tree but with both single- and double-edges allowed, and at least one double-edge required; let $\pat_1(A)$ denote the set of such multigraphs and for $\Pi \in \pat_1(A)$, write $\phi \in \emb(\Pi)$ when $\sh(\phi)$ is isomorphic to $\Pi$. The remaining term to handle is
\begin{align}
\frac{1}{\sqrt{|\emb(A)|}} \sum_{\phi \in L_1} C_\phi \bY^{\sk(\phi)} &= \frac{1}{\sqrt{|\emb(A)|}} \sum_{\Pi \in \pat_1(A)} \, \sum_{\phi \in \emb(\Pi)} C_\phi \bY^{\sk(\phi)}. \nonumber
\intertext{Note that $C_\Pi := C_\phi$ depends only on $\Pi$ (not $\phi$). Also write $\Sk(\Pi)$ for the rooted tree isomorphic to $\sk(\phi)$, which again only depends on $\Pi$. The above becomes}
&\hspace{-40pt}= \frac{1}{\sqrt{|\emb(A)|}} \sum_{\Pi \in \pat_1(A)} C_\Pi \sum_{\phi \in \emb(\Sk(\Pi))}  \bY^{\phi} \, (1+o(1)) \, n^{|V(\Pi)| - |V(\Sk(\Pi))|} \nonumber \\
&\hspace{-40pt}= \frac{1+o(1)}{\sqrt{|\emb(A)|}} \sum_{\Pi \in \pat_1(A)} C_\Pi \, n^{|V(\Pi)| - |V(\Sk(\Pi))|} \sqrt{|\emb(\Sk(\Pi))|} \, G_{\Sk(\Pi)}(\bY) \, .
\label{eq:L1-G-term}
\end{align}
The number of double-edges in $\Pi$ is $d(\Pi) = |V(A)| - |V(\Pi)|$. We have
\[ |\emb(A)| = (1+o(1)) n^{|V(A)|-1}, \]
\[ C_\Pi = (\const(\Pi)+o(1)) n^{-\frac{1}{2}(|V(\Pi)| - |V(\Sk(\Pi))| - d(\Pi))} = (\const(\Pi)+o(1)) n^{-\frac{1}{2}(2|V(\Pi)| - |V(\Sk(\Pi))| - |V(A)|)}. \]
Now~\eqref{eq:L1-G-term} becomes
\begin{align*}
&\sum_{\Pi \in \pat_1(A)} (\const(\Pi)+o(1)) n^{-\frac{1}{2}(|V(A)|-1) - \frac{1}{2}(2|V(\Pi)| - |V(\Sk(\Pi))| - |V(A)|) + |V(\Pi)| - |V(\Sk(\Pi))| + \frac{1}{2}(|V(\Sk(\Pi))|-1)} G_{\Sk(\Pi)}(\bY) \\
&\quad = \sum_{\Pi \in \pat_1(A)} (\const(\Pi)+o(1)) \, G_{\Sk(\Pi)}(\bY) \\
&\quad = \sum_{B \in \cT_{\le D} \,:\, |E(B)| < |E(A)|} (\const(A,B)+o(1)) \, G_B(\bY) \\
&\quad =: \sum_{B \in \cT_{\le D} \,:\, |E(B)| < |E(A)|} \const(A,B) \, G_B(\bY) + \PolErr_{A,4}
\end{align*}
where $\EE[\PolErr_{A,4}^2] = o(1)$ because $\EE[G_B^2] = O(1)$ (similarly to the proof of Lemma~\ref{lem:E_A2}).

\begin{proof}[Proof of Lemma~\ref{lem:change-basis}]
Summarizing the above, we have shown how to write
\[ \PolH_A = \const(A) \Poly_\emptyset - \PolErr_{A,0} + G_A \]
where
\[ G_A = \Poly_A + \PolErr_{A,1} - \PolErr_{A,2} - \PolErr_{A,3} - \PolErr_{A,4} - \sum_{B \in \cT_{\le D} \,:\, |E(B)| < |E(A)|} \const(A,B) \, G_B \, . \]
Using induction on $|E(A)|$ we can apply the same procedure to expand $G_B$ in the basis
 $\{\Poly_C : C \in \cT_{\le D}, \, |E(C)| < |E(B)|\}$ plus an error
  term. Recalling~\eqref{eq:triangle-error}, this now gives the desired expansion for $\PolH_A$.
\end{proof}

\section{Proof of Lemma \ref{lemma:SEVec}}
\label{app:LemmaSE}

Let $(\bZ(t))_{t\in\ZZ}$,  $\bZ(t)\in\reals^{n\times n}$ be 
a collection of i.i.d.\ copies of $\bZ$, and define 
 \begin{align}
 \bY(t) = \frac{1}{\sqrt{n}}\btheta\btheta^{\sT} +\bZ(t)\, .
 \end{align}
 We will use the sequence of random matrices $\bY_*:=\{\bY(t)\}$
 uniquely as a device for algorithm analysis and 
 not in the actual estimation procedure.
 
 We  extend the definition of tree structured polynomials 
 (cf.\ Eq.~\eqref{eq:F_A})
 to such sequences of random matrices via
 \begin{equation}\label{eq:F_A_Gen}
\Poly^t_{\Tree}(\bY_*) = \frac{1}{\sqrt{|\nr(\Tree)|}}
 \sum_{\phi \in \nr(\Tree)} \, 
\prod_{(i,j) \in E(\Tree)} Y_{\phi(i),\phi(j)}
(t-\dist_{\Tree}((i,j),\root))\, .
\end{equation}
Here $\dist_{\Tree}((i,j),\root) :=\max(\dist_{\Tree}(i,\root),
\dist_{\Tree}(j,\root))$.  When the argument is a single matrix $\bY$,
then $\Poly^t_{\Tree}(\bY)$ is defined by applying Eq.~\eqref{eq:F_A_Gen} to the sequence of matrices 
given by $\bY(s) =\bY$  for all $s\in\ZZ$
(thus recovering Eq~\eqref{eq:F_A}). 

The random variable $\Poly^t_{\Tree}(\bY_*)$ does depend on $t$. However 
its distribution is independent of $t$. We will therefore often omit the superscript $t$.
(In the case of $\Poly^t_{\Tree}(\bY)$, the random variable itself does not depend on $t$.)

We define a subset of the family of non-reversing 
labelings of $\Tree\in\cT_{\le D}$. 
\begin{definition}\label{def:StronglyNR}
A labeling $\phi$ of $\Tree \in \cT_{\le D}$ is said to be 
\emph{strongly non-reversing} if it is non-reversing and
for any two edges $(i,j)$, $(k,l)\in E(\Tree)$,
with $\dist_{\Tree}((i,j),\root)\neq \dist_{\Tree}((k,l),\root)$,
we have $(\phi(i),\phi(j))\neq (\phi(k),\phi(l))$ (as unordered pairs). 
We denote by  $\nrs(\Tree)$  the set of all
strongly  non-reversing labelings of $\Tree$.

A pair of labelings $\phi_1\in\nr(\Tree_1)$, $\phi_2\in(\Tree_2)$
is said to be jointly strongly non-reversing if
each of them is strongly non-reversing and  
$\dist_{\Tree_1}((i,j),\root)\neq \dist_{\Tree_2}((k,l),\root)$
implies $(\phi_1(i),\phi_1(j))\neq (\phi_2(k),\phi_2(l))$.
We denote by $\nrs(\Tree_1,\Tree_2)$ the set of such pairs.

We define the modified polynomials 
$\Polys^t_{\Tree}(\bY)$, $\Polys^t_{\Tree}(\bY_*)$
by restricting the sum  to $\nrs(\Tree)$ in Eqs.~\eqref{eq:F_A},
 \eqref{eq:F_A_Gen}. 
 
We also define
\begin{align*}
&\Poly^t_{\Tree_1,\Tree_2}(\bY_*) := \frac{1}{\sqrt{|\nr(\Tree_1)|\cdot|\nr(\Tree_2)|}}\;\cdot\\
& \sum_{(\phi_1,\phi_2) \in \nrs(\Tree_1,\Tree_2)} \, 
\prod_{(i,j) \in E(\Tree_1)} \bY_{\phi_1(i),\phi_1(j)}(t-\dist_{\Tree_1}((i,j),\root))
\prod_{(i,j) \in E(\Tree_2)} \bY_{\phi_2(i),\phi_2(j)}(t-\dist_{\Tree_2}((i,j),\root))\, .
\end{align*}
As before $\Poly^t_{\Tree_1,\Tree_2}(\bY)$ is obtained from the above definition by $\bY(s) =\bY$  for all $s\in\ZZ$.
\end{definition}

As we next see, the moving from non-reversing to
strongly non-reversing embeddings has a negligible effect.
\begin{lemma}\label{lemma:time}
Assume $\pi_{\Theta}$ to have finite moments of all order, and 
$|\psi(\theta)|\le B(1+|\theta|)^B$ for a constant $B$. Then, for any fixed 
$\Tree,\Tree_1,\Tree_2\in \cT_{\le D}$, there exist constants
$C_* = C_*(B,\Tree)$, $C_{\#} = C_{\#}(B,\Tree_1,\Tree_2)$
such that
\begin{equation}\label{eq:Strong1}
\left|\E[\psi(\theta_1)\Poly_{\Tree}(\bY)]-\E[\psi(\theta_1)\Polys_{\Tree}(\bY)]
\right|\le \frac{C_*}{\sqrt{n}}\, ,\;\;\;\;
\left|\E[\Poly_{\Tree_1}(\bY)\Poly_{\Tree_2}(\bY)]-
\E[\Polys_{\Tree_1,\Tree_2}(Y)]
\right|\le \frac{C_{\#}}{\sqrt{n}} \, ,
\end{equation}
\begin{equation}\label{eq:Strong2}
\left|\E[\psi(\theta_1)\Poly^t_{\Tree}(\bY_*)]-\E[\psi(\theta_1)\Polys^t_{\Tree}(\bY_*)]
\right|\le \frac{C_*}{\sqrt{n}}
,\;\;\;\;
\left|\E[\Poly^t_{\Tree_1}(\bY_*)\Poly^t_{\Tree_2}(\bY_*)]-
\E[\Polys^t_{\Tree_1,\Tree_2}(\bY_*)]
\right|\le \frac{C_{\#}}{\sqrt{n}}
\, .
\end{equation}
\end{lemma} 
\begin{proof}
We will prove Eq.~\eqref{eq:Strong1} since \eqref{eq:Strong2} follows by the same argument.
Considering the first bound, note that 
\begin{align}
&\left|\E[\psi(\theta_1)\Poly_{\Tree}(\bY)]-\E[\psi(\theta_1)\Polys_{\Tree}(\bY)]\right|\nonumber\\
&\le \frac{1}{\sqrt{|\nr(\Tree)|}}\sum_{\phi\in \nr(\Tree)\setminus \nrs(\Tree)}
\left|\E\Big[\psi(\theta_1) \prod_{(i,j) \in E(\Tree)} Y_{\phi_1(i),\phi_1(j)} \Big]\right|\nonumber\\
&\le  \frac{1}{\sqrt{|\nr(\Tree)|}}\sum_{\Pi\in \pat''(\Tree)}\sum_{\phi\in\emb(\Pi)}
\left|\E\Big[\psi(\theta_1) \prod_{(i,j) \in E(\Tree)} Y_{\phi_1(i),\phi_1(j)} \Big]\right|\, ,
\label{eq:SumPatterns}
\end{align}
where in the last line  $\pat''(\Tree)$ denotes the equivalence classes of 
$\nr(\Tree)\setminus\nrs(\Tree)$ under rooted graph homomorphisms.
Now the expectation in the last line only depends on $\Pi$. 
Taking expectation conditional on $\theta$, we get
\begin{align*}
\Big|\E\Big[\psi(\theta_1) \prod_{(i,j) \in E(\Tree)} Y_{\phi_1(i),\phi_1(j)} \;\Big\vert\; \theta\Big]\Big|
\le \psi(\theta_1)\prod_{\substack{1 \le i\le j\le n \\ \alpha_{ij}(\phi) = 1}}\frac{\theta_i\theta_j}{\sqrt{n}}
\prod_{\substack{1 \le i\le j\le n \\ \alpha_{ij}(\phi) \ge 2}}2^{\alpha_{ij}(\phi)-1}
\left(\frac{|\theta_i\theta_j|}{\sqrt{n}} + C\sqrt{\alpha_{ij}}\right)\, .
\end{align*}
Here $\alpha(\phi)$ is the image of embedding $\phi$.
Next taking expectation with respect to $\theta$ and using that all moments are finite
\begin{align}
\Big|\E\Big[\psi(\theta_1) \prod_{(i,j) \in E(\Tree)} Y_{\phi_1(i),\phi_1(j)} \Big]\Big|
\le C(\psi,\Pi)\, n^{-e_1(\Pi)/2}\, ,
\end{align}
where $e_1(\Pi)$ is the number of edges with multiplicity $1$ in $\Pi$.

Using this bound in Eq.~\eqref{eq:SumPatterns} and noting that $|\emb(\pi)|\le n^{v(\Pi)-1}$
(where $v(\Pi)$ is the number of vertices in  $\Pi$), we get
\begin{align*}
&\left|\E[\psi(\theta_1)\Poly_{\Tree}(\bY)]-\E[\psi(\theta_1)\Polys_{\Tree}(\bY)]\right|\\
&\le  C\frac{1}{\sqrt{|\nr(\Tree)|}}\sum_{\Pi\in \pat''(\Tree)} n^{v(\Pi)-1-e_1(\Pi)/2}\\
&\le C' \sum_{\Pi\in \pat''(\Tree)}n^{v(\Pi)-1-e_1(\Pi)/2-m(\Pi)/2}\, ,
\end{align*}
where $m(\pi)$ is the number of edges of $\Pi$ counted with their multiplicity.
In the last step we used the fact that $m(\pi)=V(\Tree)-1$ and 
$|\nr(\Tree)|\ge C_0 n^{V(\Tree)-1}$.

Next notice that, denoting by $e_{\ge 2}(\Pi)$ the number of edges in $\Pi$, this time without counting
multiplicity, we have
\begin{align*}
 \frac{1}{2}e_1(\Pi)+\frac{1}{2}m(\Pi)+1-v(\Pi)&\ge e_1(\Pi)+e_{\ge 2}(\Pi)+1-v(\Pi)\\
 & \ge \lp(\Pi)\, .
 \end{align*}
 where $\lp(\Pi)$ is the number of self-loops in the projection of $\Pi$ onto simple graphs 
 (i.e.\ the graph obtained by replacing every multi-edge in $\Pi$ by a single edge).
 For $\Pi\in\pat''(\Tree)$, we have $\lp(\Pi)\ge 1$, whence the claim follows.
 
 The proof for the second equation in Eq.~\eqref{eq:Strong1} is very similar.
  Using the shorthand $ \nr\setminus \nrs(\Tree_1,\Tree_2):=
  \nr(\Tree_1)\times \nr(\Tree_2)
  \setminus \nrs(\Tree_1,\Tree_2)$, we
   begin by writing
 \begin{align*} 
& \left|\E[\Poly_{\Tree_1}(\bY)\Poly_{\Tree_2}(\bY)]-
\E[\Polys_{\Tree_1,\Tree_2}(\bY)]
\right|\\
 &\le \frac{1}{\sqrt{|\nr(\Tree_1)|\cdot|\nr(\Tree_2)|}}
 \sum_{(\phi_1,\phi_2) \in \nr\setminus \nrs(\Tree_1,\Tree_2)}
 \Big| \E\Big\{
\prod_{(i,j) \in E(\Tree_1)} Y_{\phi_1(i),\phi_1(j)}
\prod_{(i,j) \in E(\Tree_2)} Y_{\phi_2(i),\phi_2(j)}
\Big\}\Big|\\
&\le  \frac{1}{\sqrt{|\nr(\Tree_1)| \cdot|\nr(\Tree_2)|}}
\sum_{\Pi\in \pat''(\Tree_1,\Tree_2)}
\sum_{(\phi_1,\phi_2)\in\emb(\Pi)}
\left|\E\Big[\prod_{(i,j) \in E(\Tree_1)} Y_{\phi_1(i),\phi_1(j)} 
\prod_{(i,j) \in E(\Tree_2)} Y_{\phi_2(i),\phi_2(j)}\Big]\right|\, .
\end{align*}
The only difference with respect to the previous case lies in the 
fact that $\pat''(\Tree_1,\Tree_2)$ is a collection of
graphs with edges labeled by $\{1,2\}$. (It is the set of
equivalence classes of $\nr\setminus \nrs(\Tree_1,\Tree_2)$ under
graph homomorphisms.)

Proceeding as before, 
 \begin{align*} 
\left|\E\Big[\prod_{(i,j) \in E(\Tree_1)} Y_{\phi_1(i),\phi_1(j)} 
\prod_{(i,j) \in E(\Tree_2)} Y_{\phi_2(i),\phi_2(j)}\Big]\right|
\le C(\Pi) \, n^{-e_1(\Pi)}\, .
\end{align*}
Therefore
 \begin{align*} 
& \left|\E[\Poly_{\Tree_1}(\bY)\Poly_{\Tree_2}(\bY)]-
\E[\Polys_{\Tree_1,\Tree_2}(\bY)]
\right|\le \frac{1}{\sqrt{|\nr(\Tree_1)| \cdot|\nr(\Tree_2)|}}
\sum_{\Pi\in \pat''(\Tree_1,\Tree_2)}
\sum_{(\phi_1,\phi_2)\in\emb(\Pi)} C(\Pi)\, 
 n^{-e_1(\Pi)} \, .
 \end{align*}
Recall that $|\nr(\Tree_1)|\ge C_0 n^{V(\Tree_1)-1}$,
$|\nr(\Tree_2)|\ge C_0 n^{V(\Tree_2)-1}$ and, as before,
$V(\Tree_1)+V(\Tree_2)-2 = m(\Pi)$. 
Further $|\emb(\Pi)|\le C_1 \, n^{v(\Pi)-1}$, whence
 \begin{align*} 
& \left|\E[\Poly_{\Tree_1}(\bY)\Poly_{\Tree_2}(\bY)]-
\E[\Polys_{\Tree_1,\Tree_2}(\bY)]
\right|\le C_2 \, \sum_{\Pi\in \pat''(\Tree_1,\Tree_2)}
n^{v(\Pi)-1-e_1(\Pi)/2-m(\Pi)/2}\, .
\end{align*}
The last sum is upper bounded by $C\, n^{-1/2}$ by the same argument as above.
\end{proof}

Consider now a  term corresponding to $\phi\in\nrs(G)$ of the expectation
$\E[\psi(\theta_1)\Polys_{\Tree}(Y_*) \, | \, \theta]$. By construction, this does not involve
moments $\E[Y^a_{ij}(t_1)Y^b_{ij}(t_2)\cdots \,|\, \theta]$ for $t_1\neq t_2$.
Therefore the expectations coincide when considering $\Polys_{\Tree}(\bY_*)$
or $\Polys_{\Tree}(Y)$. In other words, we have the identities
\begin{align}
\E[\psi(\theta_1)\Polys_{\Tree}(\bY_*)] = \E[\psi(\theta_1)\Polys_{\Tree}(\bY)]\, ,
\;\;\;\;\; \E[\Polys_{\Tree_1,\Tree_2}(\bY_*)] = \E[\Polys_{\Tree_1,\Tree_2}(\bY)]\, .
\end{align}
We therefore proved the following consequence of Lemma \ref{lemma:time}.
\begin{corollary}\label{coro:YYstar}
Under the assumptions of Lemma \ref{lemma:time}, there exists $C_0=C_0(\Tree,B)$
such that
\begin{align}
\left|\E[\psi(\theta_1)\Poly_{\Tree}(\bY)]-\E[\psi(\theta_1)\Poly_{\Tree}(\bY_*)]
\right|\le \frac{C_0}{\sqrt{n}}\, 
\, .
\end{align}
\end{corollary}
We now consider a message passing of the same form as Eq.~\eqref{eq:MP1},
but with $\bY$ replaced by $\bY(t)$ at iteration $t$. Namely we define
\begin{align}
\bms_{i\to j}^{t+1}&= \frac{1}{\sqrt{n}}\sum_{k\in [n]\setminus \{i,j\}}Y_{ik}(t)F_t(\bms^t_{k\to i})\, , 
\;\;\;\;\; \, \;\; \bms_{i\to j}^0  =\E[\Theta]\;\;\forall i\neq j\, 
 ,\label{eq:MP1T}
\end{align}
and:
\begin{align}
\bms_i^{t+1} &= \frac{1}{\sqrt{n}}\sum_{k\in [n]\setminus \{i\}}
Y_{ik}(t)F_t(\bms^t_{k\to i}) \, ,\\
\hat{\bms}_{i}^{t+1}&= F_t(\bms^{t+1}_{i})\, .\label{eq:MP2T}
\end{align}
Analogously to the case of the original iteration, we can write these 
as sums over polynomials $\Poly^{t+1}_{\Tree}(\bY_*)$, with coefficients that have a limit as 
$n\to\infty$.
Therefore, we have a further consequence of Lemma \ref{lemma:time} and
the previous corollary.
\begin{corollary}\label{coro:MM_MS}
Assume $\pi_{\Theta}$ to have finite moments of all order, and 
$\psi:\reals^{\dim+1}\to\reals$ be a polynomial with coefficients bounded by $B$
and maximum degree $B$.

Then there exists $C_0=C_0(T,B)$
such that, for all $t\le T$,
\begin{align}
&\left|\E[\psi(\theta_1,\bmm_1^{t})]-\E[\psi(\theta_1,\bms_1^{t})]
\right| \le \frac{C_0}{\sqrt{n}}\, ,\\
&\left|\E[\psi(\theta_1,\bmm_{1\to 2}^{t})]-\E[\psi(\theta_1,\bms_{1\to 2}^{t})]
\right|\le \frac{C_0}{\sqrt{n}}\, .
\end{align}
\end{corollary}
\begin{proof} 
We expand the polynomial $\psi$ as
\begin{align*}
\psi(\theta,\bx) = \sum_{\bm\in\naturals^{\dim}}\psi_{\bm}(\theta) \bx^{\bm}\, ,\;\;\;\;\;
\bx^{\bm} := \prod_{i=1}^{\dim}x_i^{m_i}\, .
\end{align*}  
Further, $\bmm^t_1$ can be expressed as a sum over $\Tree\in \cT_{\le D}$ and
therefore the same holds for $(\bmm^t_1)^{\bm}$ 
\begin{align*}
(\bmm^t_1)^{\bm} &= \sum_{(\Tree_{\ell,k})_{\ell\le \dim,k\le m_{\ell}}}
\prod_{\ell\le \dim,k\le m_{\ell}}c_{\Tree_{\ell,k}} \prod_{\ell\le \dim,k\le m_{\ell}}
\Poly_{\Tree_{\ell,k}} (\bY)\, ,\\
&= \sum_{\Tree}\overline{c}_{\Tree}(\bm)  \Poly_{\Tree} (\bY)\, ,
\end{align*}
where, in the last row, we grouped terms such that $\cup_{\ell,k}\Tree_{\ell,k}=\Tree$
and (up to combinatorial factors which are independent of $n$), the coefficient $\overline{c}_{\Tree}(\bm)$ is
given by the product  of coefficients $c_{\Tree_{\ell,k}}$. 

Hence
\begin{align*}
\Big|\E[\psi(\theta_1,\hmm_1^{t})]-\E[\psi(\theta_1,\hms_1^{t})]\Big|
\le \sum_{\bm}\sum_{\Tree} \overline{c}_{\Tree}(\bm) 
\Big| \E[\psi_{\bm}(\theta_1)\Poly_{\Tree} (\bY)]-\E[ \psi_{\bm}(\theta_1)\Poly^t_{\Tree} (\bY_*)]\Big|\, ,
\end{align*}
and the claim follows by noting that the sums over $\bm$ and $\Tree$
involve a constant (in $n$) number of terms, the coefficients 
$\overline{c}_{\Tree}(\bm)$ have a limit as $n\to\infty$, and applying
Corollary \ref{coro:YYstar}.
\end{proof}

\begin{lemma}\label{lemma:BoundedMoments}
Assume $\pi_{\Theta}$ to have finite moments of all order, and let $m_0$, $\bm\in\naturals^\dim$
be fixed. Then there exists a constant $C= C(t;\bm)$ independent of $n$ such
that
\begin{align}
\big|\E[(\bms_{i}^t)^{\bm}]\big|\le C\, ,\;\;\;\;\big|\E[(\bms_{i\to j}^t)^{\bm}]\big|\le C\, .
\end{align}
\end{lemma}

As a final step towards the proof of Lemma \ref{lemma:SEVec}, we prove the analogous statement
for the modified iteration $(\hms^t_i)$.
\begin{lemma}\label{lemma:MS_SE}
Assume $\pi_{\Theta}$ to have finite moments of all order, and 
$\psi$ to be a fixed polynomial. 
Define the sequence of vectors $\bmu_t\in\reals^{\dim}$ and positive semidefinite matrices
$\bSigma_t\in \reals^{\dim\times \dim}$ via the state evolution equations 
\eqref{eq:GeneralSE1}, \eqref{eq:GeneralSE2}.

Then the following hold (here $\plim$ denotes limit in probability):
\begin{align}
\lim_{n\to\infty}
\E[\psi(\bms^t_1,\theta_1)]= 
\E\psi\big(\bmu_t\Theta+\bG_t,\Theta\big)\, ,\;\;\;\;
\plim_{n\to\infty}\frac{1}{n}\sum_{i=1}^n\psi(\bms^t_{i},\theta_i)= 
\E\psi\big(\bmu_t\Theta+\bG_t,\Theta\big)\, ,\label{eq:T-SEVec1-Convergence-Almost}\\
\lim_{n\to\infty}
\E[\psi(\bms^t_{1\to 2},\theta_1)]= 
\E\psi\big(\bmu_t\Theta+\bG_t,\Theta\big)\, ,\;\;\;\;
\plim_{n\to\infty}\frac{1}{n}\sum_{i=2}^n\psi(\bms^t_{i\to 1},\theta_i)= 
\E\psi\big(\bmu_t\Theta+\bG_t,\Theta\big)\, .\label{eq:T-SEVec2-Convergence-Almost}
\end{align}
\end{lemma}
\begin{proof}
The proof is essentially the same as the proof of Proposition 4 in \cite{bayati2015universality}.
We will focus on the claim \eqref{eq:T-SEVec2-Convergence-Almost}, since 
Eq.~\eqref{eq:T-SEVec2-Convergence-Almost} is completely analogous.

We proceed by induction over $t$, and will denote by $\cF_t$ the 
$\sigma$-algebra generated by $\theta$ and $Z(1),\dots Z(t)$. It is convenient to define 
$W(s)=Z(s)/\sqrt{n}$ and do rewrite Eq.~\eqref{eq:MP1T}
as
\begin{align}
\bms_{i\to j}^{t+1}= 
\Big(\frac{1}{n}\sum{k\in [n]\setminus \{i,j\} } \theta_k F_t(\bms^t_{k\to i})\Big)
\theta_i + \sum_{k\in [n]\setminus \{i,j\}}W_{ik}(t)F_t(\bms^t_{k\to i})\, .
\end{align}
Fixing $i,j$, by the induction hypothesis 
\begin{align}
&\plim_{n\to\infty}\frac{1}{n}\sum_{k\in [n]\setminus \{i,j\} } \theta_k F_t(\bms^t_{k\to i}) 
=\E\big\{\Theta F_t(\bmu_t\Theta+\bG_t)\big\}= \bmu_{t+1}\, ,\\
&\plim_{n\to\infty}\frac{1}{n}\sum_{k\in [n]\setminus \{i,j\} } 
F_t(\bms^t_{k\to i}) F_t(\bms^t_{k\to i})^{\sT}
=\E\big\{F_t(\bmu_t\Theta+\bG_t)F_t(\bmu_t\Theta+\bG_t)^{\sT}\big\} = \bSigma_{t+1}\, ,\\
&\plim_{n\to\infty}\frac{1}{n}\sum_{k\in [n]\setminus \{i,j\} } 
\|F_t(\bms^t_{k\to i})\|^4 = C_t<\infty\, .
\end{align}
Construct $(\theta_i)_{i\ge 1}$ for different $n$ in the same probability space.
By Lyapunov's central limit theorem, we have that, in probability 
\begin{align}
{\sf Law}(\bms_{i\to j}^{t+1}|\cF_t) \Rightarrow \normal(\bmu_{t+1}\theta_i,\bSigma_{t+1})\, .
\end{align}
(Here $\Rightarrow$ denotes weak convergence of probability measures.)
Hence, for any bounded Lipschitz function $\overline\psi$, we have
\begin{align}
\lim_{n\to\infty}
\E[\overline{\psi}(\bms^{t+1}_{i\to j},\theta_i)|\cF_t]= 
\E\overline\psi\big(\bmu_{t+1}\Theta+\bG_{t+1},\Theta\big)\, .
\end{align}
Since the right hand side is non-random, and using dominated convergence,
we have
\begin{align}
\lim_{n\to\infty}
\E[\overline{\psi}(\bms^{t+1}_{i\to j},\theta_i)]= 
\E\overline\psi\big(\bmu_{t+1}\Theta+\bG_{t+1},\Theta\big)\, .
\end{align}
Applying Lemma \ref{lemma:BoundedMoments}, the claim also holds
for $\psi$ that is only polynomially bounded, thus proving the
the first equation in Eq.~\eqref{eq:T-SEVec2-Convergence-Almost}, for iteration $t+1$.

Next consider the second limit in  Eq.~\eqref{eq:T-SEVec2-Convergence-Almost},
and begin by considering the case in which $\psi$ is a fixed polynomial.
We claim that, for any fixed $i_1\neq i_2\in\{2,\dots,n\}$,
\begin{align}
\lim_{n\to\infty}
\big|\E[\psi(\theta_{i_1},\bms_{i_1\to 1}^{t+1})\psi(\theta_{i_2},\bms_{i_2\to 1}^{t+1})]-
\E[\psi(\theta_{i_1},\bms_{i_1\to 1}^{t+1})]\E[\psi(\theta_{i_2},\bms_{i_2\to 1}^{t+1})]\big|=0\, .
\label{eq:CovIsSmall}
\end{align}
Indeed by linearity it is sufficient to prove that this is the case for 
$\psi(\theta,\bx) =\psi_{\bm}(\theta)\bx^{\bm}$. This case in turn can be analyze by 
expanding $(\bms_{i\to 1}^t)^{\bm}$ in a sum over trees as in the proof of Lemma~\ref{lemma:time}.
(See Proposition 4 in \cite{bayati2015universality}.) 

By expanding the sum in the variance and using Eq.~\eqref{eq:CovIsSmall}
we thus get
\begin{align*}
\lim_{n\to\infty}{\rm Var}\left(\frac{1}{n}\sum_{i=2}^n\psi(\bms^{t+1}_{i\to 1},\theta_i)\right)=0\, .
\end{align*}
Further we already proved earlier that
\begin{align*}
\lim_{n\to\infty}\E\left(\frac{1}{n}\sum_{i=2}^n\psi(\bms^{t+1}_{i\to 1},\theta_i)\right)=
\E\overline\psi\big(\bmu_{t+1}\Theta+\bG_{t+1},\Theta\big)\, .
\end{align*}
Hence, by Chebyshev's inequality, the following holds for any polynomial $\psi$:
\begin{align*}
\plim_{n\to\infty}\frac{1}{n}\sum_{i=2}^n\psi(\bms^{t+1}_{i\to 1},\theta_i)= 
\E\psi\big(\bmu_{t+1}\Theta+\bG_{t+1},\Theta\big)\, .
\end{align*}
This completes the proof of Lemma~\ref{lemma:MS_SE}.
\end{proof}

Now Lemma \ref{lemma:SEVec} is an immediate consequence of 
Corollary~\ref{coro:MM_MS} and Lemma~\ref{lemma:MS_SE}.

\bibliographystyle{alpha}

\newcommand{\etalchar}[1]{$^{#1}$}

\end{document}